\documentclass[12pt]{article}
\usepackage{amssymb}
\usepackage{amsmath,bm,natbib,graphics,mathrsfs,graphicx,epsfig,subfigure,placeins,indentfirst,setspace,enumerate,enumitem,mhchem,caption,varioref,booktabs,tabularx,color,epstopdf,multirow}
\setcitestyle{authoryear,open={(},close={)}}
\usepackage{amsthm}
\makeatletter \oddsidemargin  -.1in \evensidemargin -.1in
\begin{document}
	\newcommand{\bea}{\begin{eqnarray}}
		\newcommand{\eea}{\end{eqnarray}}
	\newcommand{\nn}{\nonumber}
	\newcommand{\bee}{\begin{eqnarray*}}
		\newcommand{\eee}{\end{eqnarray*}}
	\newcommand{\lb}{\label}
	\newcommand{\nii}{\noindent}
	\newcommand{\ii}{\indent}
	\newtheorem{theorem}{Theorem}[section]
	\newtheorem{example}{Example}[section]
	\newtheorem{corollary}{Corollary}[section]
	\newtheorem{definition}{Definition}[section]
	\newtheorem{lemma}{Lemma}[section]
	\newtheorem{remark}{Remark}[section]
	\newtheorem{proposition}{Proposition}[section]
	\numberwithin{equation}{section}
	\renewcommand{\qedsymbol}{\rule{0.7em}{0.7em}}
	\renewcommand{\theequation}{\thesection.\arabic{equation}}
	\renewcommand\bibfont{\fontsize{10}{12}\selectfont}
	\setlength{\bibsep}{0.0pt}
		\title{\bf Copula-based extropy measures, properties and dependence in bivariate distributions}
	
			\author{Shital {\bf Saha}\thanks {Email address: shitalmath@gmail.com,~520MA2012@nitrkl.ac.in}~~and~Suchandan {\bf Kayal}\thanks{Corresponding author : Suchandan Kayal (kayals@nitrkl.ac.in,~suchandan.kayal@gmail.com)}}

	\maketitle \noindent {\it Department of Mathematics, National Institute of
		Technology Rourkela, Rourkela-769008, Odisha, India} 
	\date{}
\begin{center}
\textbf{Abstract}
\end{center}
In this work, we propose extropy measures based on density copula, distributional copula, and survival copula, and explore their properties. We study the effect of monotone transformations for the proposed measures and obtain some bounds. We establish connections between cumulative copula extropy and three dependence measures: Spearman's rho, Kendall's tau, and Blest's measure of rank correlation. Finally, we propose  estimators for the cumulative copula extropy and survival copula extropy with an illustration using real life datasets.

\section{Introduction}	
The extropy, introduced by \cite{lad2015extropy} is a complementary dual of the Shannon entropy. Let $X$ be a discrete random variable (RV) with probability mass function $\{P(X=x_{i})=p_{i}>0;~i=1,\ldots,n\}.$ Then, the Shannon entropy of $X$ is (see \cite{shannon1948mathematical})
\begin{eqnarray}\label{eq1.1}
I(X)=-\sum_{i=1}^{n}p_{i}\log (p_{i}),
\end{eqnarray}
where $\log(\cdot)$ represents the logarithmic function with base $e$. The extropy of the discrete RV $X$ is given by (see \cite{lad2015extropy})
\begin{eqnarray}\label{eq1.2}
I^*(X)=-\sum_{i=1}^{n}(1-p_{i})\log (1-p_{i}).
\end{eqnarray}
 The measures given by (\ref{eq1.1}) and (\ref{eq1.2}) quantify uncertainty in contrasting styles. It can be noted from (\ref{eq1.1}) and (\ref{eq1.2}) that $I(X)=I^{*}(X)$, for $n=2.$ However, when $n>2$, we have $I(X)>I^*(X)$, that is the entropy is larger than the extropy. Similar to Shannon's axioms, the extropy is continuous with respect to its arguments. Further, $I^*(\frac{1}{n},\ldots,\frac{1}{n})$ is monotonic increasing function in $n$. There are various other properties which are shared by both entropy and extropy. For example, the extropy is invariant under permutation as well as under monotonic transformations. The uniform distribution is the maximum extropy distribution, for any size of $n$. 
 
 The concepts of entropy and extropy have been developed for the continuous type RV too in the literature. Let $X$ be a non-negative absolutely continuous RV with probability density function (PDF) $f(\cdot)$. Then, the differential entropy of $X$ is (see \cite{shannon1948mathematical})
 \begin{eqnarray}
 S(X)=-\int_{0}^{\infty}f(x)\log (f(x)) dx=E\big(-\log \big(f(X)\big)\big),
\end{eqnarray}
presenting the expected information content in a RV. The differential entropy may take negative values although the entropy given by (\ref{eq1.1}) is always non-negative. A distribution with negative differential entropy is uniform distribution in the interval $(0,a)$, with $a<1.$ The differential extropy of $X$ is developed in the style suggested by \cite{shannon1948mathematical} as
\begin{eqnarray}
J(X)=-\frac{1}{2}\int_{0}^{\infty}f^{2}(x)dx=-\frac{1}{2}E(f(X)).
\end{eqnarray}
For details, we refer to \cite{lad2015extropy}. The extropy is always negative. We recall that the extropy is useful in scoring the forecasting distribution. \cite{gneiting2007strictly} showed that in the total log scoring method, the expected score of a forecasting distribution equals the negative sum of the entropy and extropy of this distribution. Further, it can be observed that there is a close relationship between the extropy and the informational energy, defined for a RV $X$ as $IE(X)=\int_{0}^{\infty}f^{2}(x)dx.$ We refer to \cite{onicescu1966theorie} for some properties of the informational energy. We observe that exploration of properties of extropy just like entropy has received a considerable interest from various authors in recent past. Few recent useful references in this direction are \cite{gupta2023some}, \cite{nair2023some},  and \cite{toomaj2023extropy}.

Consider two non-negative absolutely continuous RVs $X$ and $Y$ with joint PDF $f_{X,Y}(\cdot,\cdot).$ Then, the differential joint entropy  is given by 
 \begin{eqnarray}\label{eq1.5}
S(X,Y)=-\int_{0}^{\infty}\int_{0}^{\infty}f_{X,Y}(x,y)\log (f_{X,Y}(x,y)) dx dy=E\big(-\log (f_{X,Y}(X,Y))\big).
\end{eqnarray}
The joint entropy given by (\ref{eq1.5}) provides the amount of uncertainty contained in the RVs $X$ and $Y$, when they are observed together. For more than two RVs, the joint entropy  can be defined similarly. In an analogous manner, the joint extropy (also known as the bivariate extropy) of the pair of RVs $(X,Y)$ is defined as follows (see \cite{balakrishnan2022weighted}):
\begin{eqnarray}\label{eq1.6}
J(X,Y)=\frac{1}{4}\int_{0}^{\infty}\int_{0}^{\infty}f_{X,Y}^{2}(x,y)dxdy=\frac{1}{4}E(f_{X,Y}(X,Y)).
\end{eqnarray}
The joint extropy for $n$-dimensional random vector $(X_{1},\ldots,X_{n})$, denoted by $J(X_{1},\ldots,X_{n})$ can be defined using the joint PDF of $n$ RVs. In this case, we need to multiply the integral by the factor $(-\frac{1}{2})^{n}$. Further, we note that when the RVs $X$ and $Y$ are independent, we have $J(X,Y)=J(X)J(Y)$. We refer to \cite{balakrishnan2022weighted} for some discussions about the joint extropy measure in (\ref{eq1.6}). Thus, it is clear that to develop the concept of extropy for univariate, bivariate, and multivariate probability distributions, we need to use univariate, bivariate, and multivariate distributions in its formula. 

The probability models for dependent bivariate or multivariate observations can be specified by the joint cumulative distribution function (CDF).
Using the result of \cite{sklar1959fonctions}, if $F_{X,Y}(\cdot,\cdot)$ denotes the joint CDF of the RVs $X$ and $Y$, with marginal CDFs $F_{X}(\cdot)$ and $F_{Y}(\cdot)$, respectively, then 
\begin{eqnarray}
F_{X,Y}(x,y)=C(F_{X}(x),F_{Y}(y))=C(u,v),~u,v\in[0,1],
\end{eqnarray}
where $C(\cdot,\cdot)$ is known as the (distributional) copula function and $u=F_{X}(\cdot)$, $v=F_{Y}(\cdot)$. Thus, a multivariate (here bivariate) distribution function can be coupled with its marginal distribution functions. We recall that a function $C:[0,1]\times[0,1]\rightarrow[0,1]$ is called a bivariate copula if it satisfies 
\begin{itemize}
\item $C(0,r)=0=C(r,0)$ and $C(1,r)=r=C(r,1),$ for all $r\in[0,1]$;
\item  $C(u_2,v_2)-C(u_2,v_1)-C(u_1,v_2)+C(u_1,v_1)\ge0,$ for all $u_1,u_2,v_1,v_2\in[0,1]$ such that $u_1\le u_2$ and $v_1\le v_2$.
\end{itemize}
We refer to \cite{nelsen2006introduction} for properties and applications of the copula.  We recall that copulas are the starting point, used to construct
families of bivariate distributions. \cite{zhang2008nonparametric} explained copulas as a useful tool for modelling dependency structure. Copulas can be applied in various areas, such as the probabilistic evaluation of flood risks (see \cite{zhang2013copula}), the
structure of dependence between foreign currency and stock markets (see \cite{wang2013revisit}), and simulation studies of sea storms (see \cite{corbella2013simulating}).
Several researchers introduced information measures, coherent systems, approximate likelihood with proxy variables for parameter
estimation in high-dimensional factor and studied various properties based on copula. Further, readers may refer to \cite{schepsmeier2014derivatives},  \cite{genest2014tests}, \cite{navarro2018distribution}, \cite{chen2022goodness}, \cite{kato2022copula}, and  \cite{pumi2023novel} for some more properties and application of copula.

 In probability and statistics, knowing multivariate dependencies  in a random vector is a challenging problem. Usually, it happens due to the mixture of different relationship in the variables. It is shown by \cite{james2017multivariate} that the Shannon entropy is unable to capture such multivariate dependencies. As a result, the concept of copula entropy has been considered by various authors, see for example \cite{ma2011mutual}, \cite{zhao2011copula}, \cite{chen2013measure}, \cite{singh2018copula}, and \cite{sun2019copula}. Note  that the copula entropy is a combination of the concept of entropy and copula. For a bivariate random vector $(X,Y)$, the copula entropy is given by 
\begin{eqnarray}\label{eq1.8}
S_{c}(X,Y)=-\int_{0}^{1}\int_{0}^{1}c(u,v)\log (c(u,v))du dv,
\end{eqnarray}
where $c(\cdot,\cdot)$ is the density function corresponding to the copula function $C(\cdot,\cdot)$. Several researchers have studied copula entropy. \cite{chen2013measure} used copula entropy to measure the correlation between river flows.  \cite{chen2014copula} proposed artificial neural network model  for rainfall-runoff simulation  using copula entropy.  \cite{wang2020image} introduced  nonsubsampled contourlet transform hidden Markov tree model based on copula entropy for capturing anisotropy and directional features of images. \cite{tabatabaei2022ranking} used copula entropy for  investigating the rain-gauge network and rank the rain-gauge stations. 

\cite{hosseini2021discussion} introduced co-copula and dual copula inaccuracy measures and discussed their various properties. Recently, \cite{sunoj2023survival} introduced survival copula entropy  of a bivariate random vector $(X,Y)$ based on survival copula $\bar C(\cdot,\cdot),$ which is given by
\begin{align}\label{eq1.9}
S_{\bar C}(X,Y)=-\int_{0}^{1}\int_{0}^{1}\bar C(u,v)\log (\bar C(u,v))dudv.
\end{align}
It is remarked that using survival copula, the joint survival function is coupled to its marginal survival functions. \cite{sunoj2023survival} studied various properties of the measure in (\ref{eq1.9}) with dependence in bivariate distributions.

To the best of our knowledge, there is no work available in the literature on the present topic. In this work, we propose the notion of copula extropy (henceforth abbreviated as CEx), a complementary dual of copula entropy and study different properties. In addition to the copula extropy, we have also introduced two other information measures, say cumulative copula extropy (CCEx) and survival copula extropy (SCEx) associated with copula and survival copula, respectively. We recall that the copula models have some attractive properties and advantages over probability models, which are provided below.
\begin{itemize}
	\item The copula models combine any univariate marginal distributions, not necessarily from the same class of distributions.
	\item It is known that the complexity of a multivariate probability model increases with respect to the dimensions. However, there exist copula models (elliptic copula), for which the complexity increases at slower  rate than the probability models. 
	\item The copula models are useful to encompass several existing multivariate models. Further, one can employ the copula models for generating many other probability models.
	\item The copula gives way more options when explaining relationship between different variables, since they do not restrict the dependence structure to be linear.
\end{itemize}

The major novelties of the present work are mentioned as follows through several points.
\begin{itemize}
\item We propose new information measures based on density copula, distributional copula, and survival copula, and discuss their various properties. An inequality between the copula entropy and copula-based extropy has been obtained. We study the effects of monotone transformations for the proposed measures.
\item We have proposed horizontal, vertical, and diagonal copula extropy measures based on sections of copula and discussed properties. We have obtained an upper bound of the diagonal copula extropy. We have also proposed dual and co-copula extropy and established relations with CCEx and SCEx, respectively.
\item Further, we have discussed some properties of CCEx for PQD and NQD random vector. The relations between well known dependence measures (Spearman's rho, Kendall's tau and Blest's measure of rank correlation) with  proposed measure CCEx have been obtained.
\item Finally, we have introduced empirical CCEx and  empirical SCEX and illustrate examples for implementation in practical problems.  
\end{itemize}  

The  sequence of the presentation is as follows. In Section $2$, we propose CEx and discuss the effects under monotone transformation. In Section $3$, we propose CCEx, horizontal, vertical and diagonal copula extropy and also introduce dual copula extropy and discusses various properties. In Section $4$, we introduce SCEx and co-copula extropy, and study some of their properties. The relations of CCEx with Spearman's rho, Kendall's tau, and Blest's measure of rank correlation have been obtained in Section $5$. In Section $6$, we propose resubstitution estimators for CCEx and SCEx and then apply them in a real data set. Finally, Section $7$ concludes the paper.

 Throughout the paper, the random variables are assumed to be non-negative and absolutely continuous. The derivatives and integrations whenever used are assumed to exist.

\section{Copula extropy}
Here, we propose a new uncertainty measure associated with a density copula, dubbed as copula extropy, which can be used as an alternative of risk measure. The definition of the bivariate CEx is provided below.
\begin{definition}\label{de2.1}
Let $C(\cdot,\cdot)$ be the copula function of a random vector $(X,Y)$. Then, the CEx is defined as 
\begin{eqnarray}\label{eq2.1}
J_c(X,Y)=\frac{1}{4}\int_{0}^{1}\int_{0}^{1}c^2(u,v)dudv=\frac{1}{4}E(c(U,V)),
\end{eqnarray}
where $c(u,v)=\frac{d^2C(u,v)}{dudv}$ is the bivariate density copula corresponding to $C(u,v)$.
\end{definition}

\begin{remark}
	The information measure in (\ref{eq2.1}) can also be defined for a general $k$-dimensional random vector. In this case, it is  needed to replace the bivariate copula density by a multivariate copula density in (\ref{eq2.1}) with a multiplying factor $(-\frac{1}{2})^{k}$ in place of $\frac{1}{4}$.
\end{remark}

We note that the copula extropy in (\ref{eq2.1}) is clearly non-negative, which is one of the key properties of a risk measure.  Next, we obtain closed-form expressions of the copula extropy for some well-known copula functions, such as product copula, iterated Farlie Gumbel Morgenstern (FGM) copula (see p. 42, \cite{lai2009continuous}) and Nelsen's polynomial copula (see p. 43, \cite{lai2009continuous}) in Table $1$.

\begin{table}[h!]
	\caption {Closed-form expressions of the copula extropy of some well-known copula functions. }
	\centering 
	\scalebox{0.74}{\begin{tabular}{c c c c c c c c } 
			\hline\hline\vspace{.1cm} 
			Name & Copula function $(C(u,v))$ & Copula extropy ($J_c(X,Y)$) \\
			\hline
			Product copula& $uv$ &$\frac{1}{4}$	\\	[2EX]
			{Iterated FGM copula} &	$[\alpha (1-u)(1-v)+\beta uv(1-u)(1-v)+1]uv, ~|\alpha|\le 1,~ |\beta|\le1$	& $\frac{1}{4}[1+\frac{\alpha^2}{9}+\frac{\alpha \beta}{18}+\frac{4\beta^2}{225}]$\\[2EX]	
			Nelsen's polynomial copula & $[1+2\alpha (1-u)(1-v)(1+u+v-2uv)]uv,~0\le \alpha\le \frac{1}{4}$ & $\frac{1}{4}[1-8\alpha+ \frac{50938}{1575}\alpha^2]$\\[1EX]
			\hline
	\end{tabular}} 
	\label{tb1} 
\end{table}

Note that the FGM copula is a particular case of the iterated FGM copula. Thus, the CEx for FGM copula can be easily obtained from the expression of the CEx of iterated FGM copula by taking $\beta=0$. Now, we plot the copula extropy for FGM and Nelsen's polynomial copulas in Figures $1(a)$ and $1(b)$, respectively. The surface plot of the CEx of iterated FGM copula is presented in Figure $1(c)$. The graphs in Figures $1(a,b,c)$ establish that the CEx is not monotone with respect to its parameter(s) in general. 

\begin{figure}[h!]\label{fig2}
	\centering
	\subfigure[]{\label{c1}\includegraphics[height=1.26in]{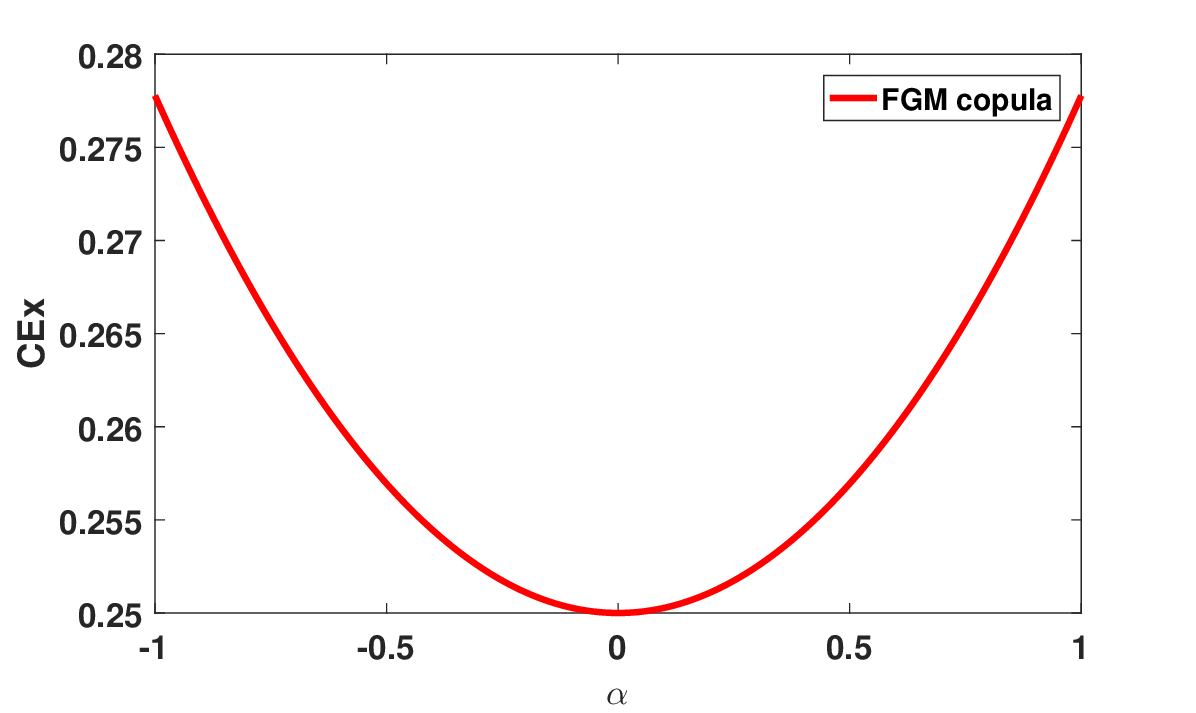}}
	\subfigure[]{\label{c1}\includegraphics[height=1.26in]{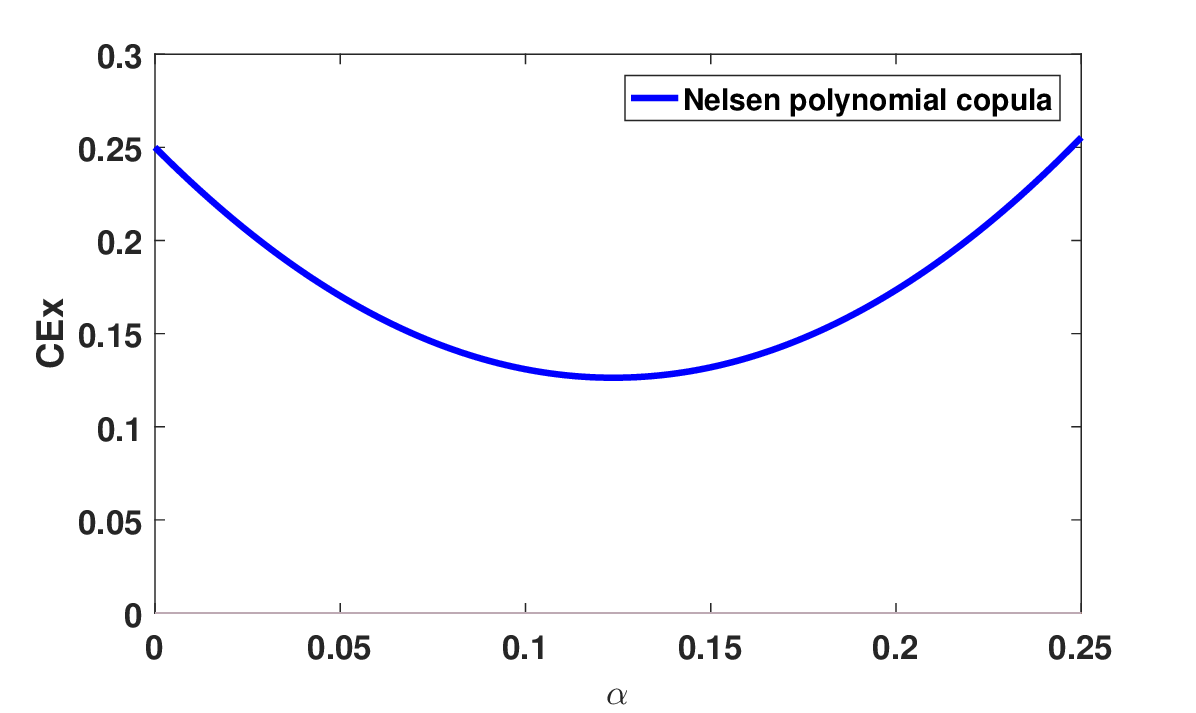}}
	\subfigure[]{\label{c1}\includegraphics[height=1.26in]{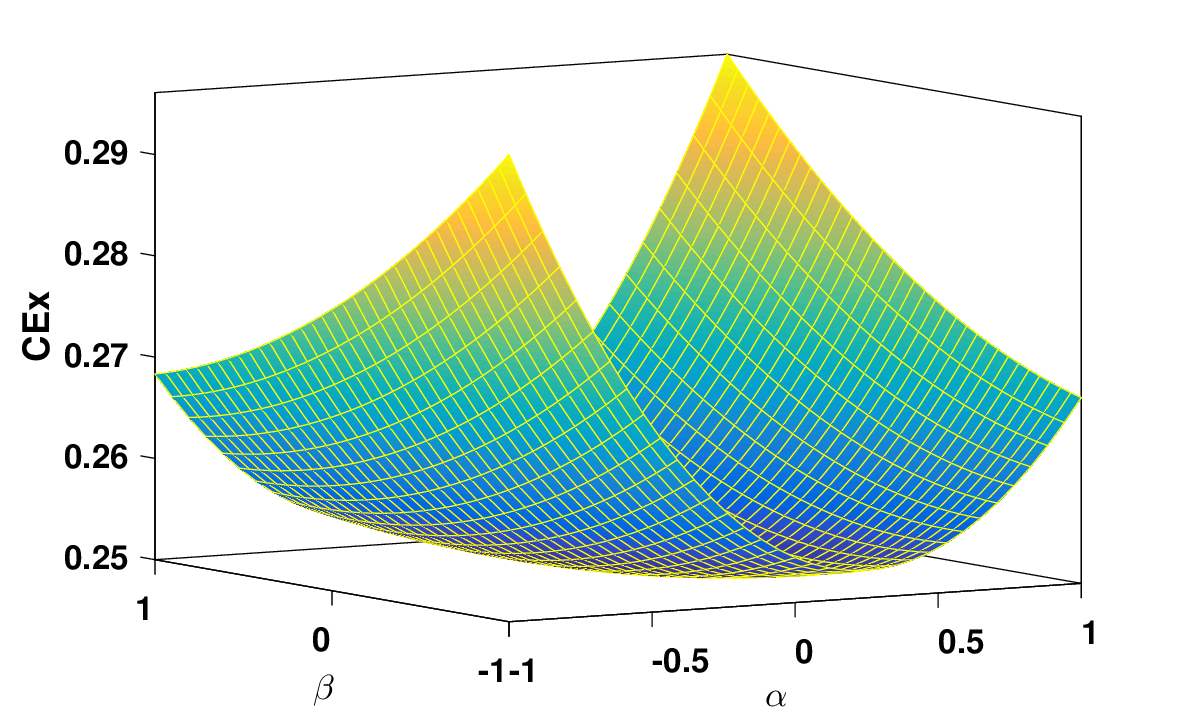}}
	\caption{ Graphs for CEx of $(a)$ FGM and $(b)$ Nelsen's polynomial copula functions. $(c)$ Surface plot  of the  CEx of iterated FGM copula.}
	\end{figure}

Copula entropy is a well-established concept proposed recently by \cite{ma2011mutual}. Thus, it would be interesting to establish connection between the copula entropy given by (\ref{eq1.8}) and the CEx in  (\ref{eq2.1}). The following proposition throws some insightful lights in this direction. Specifically, it gives a lower bound of the CEx in terms of the copula entropy. 

\begin{proposition}\label{pro2.1}
Suppose $C(\cdot,\cdot)$ is the copula function of $(X,Y)$. Then,
\begin{align*}
J_c(X,Y)\ge \frac{1}{4}\Big[1-S_c(X,Y)\Big].
\end{align*}
\end{proposition}

\begin{proof}
	For $x>0$, we have $ x\log (x)\le x^2-x$. Using the copula density function, we obtain
	\begin{align}\label{eq2.3}
	c(u,v)\log \big(c(u,v)\big)\le c^2(u,v)-c(u,v).
	\end{align}
	After multiplying $-1$ and  taking double integration with respect to $u$ and $v$ in (\ref{eq2.3}), it becomes
	\begin{align}\label{eq2.4}
	 -\int_{0}^{1}\int_{0}^{1}c(u,v)\log \big(c(u,v)\big)dudv\ge \int_{0}^{1}\int_{0}^{1}\{c(u,v)-c^2(u,v)\}dudv,
	\end{align}
	from which the result readily follows, completing the proof of the proposition.
\end{proof}

It can be proved that the mutual information between two random variables is always non-negative. Further, it has been established by \cite{ma2011mutual} that the copula entropy is actually equal to the negative of the mutual information, implying the copula entropy is always negative. This observation can also be  proved using the inequality 
$$x-1\le x \log (x),~x>0.$$
Next, we study copula extropy under monotone transformations. The following result (see Theorem $2.4.3$ and Theorem $2.4.4$ of \cite{nelsen2006introduction}) is useful in this direction. We use the abbreviations s.i. and s.d. for strictly increasing and strictly decreasing, respectively, in the rest of the paper.
 \begin{equation}\label{eq2.4*}
C_{\phi(X)\psi(Y)}(u,v)=\left\{
\begin{array}{ll}
C(u,v),~\text{if  $\phi$ and $\psi$  are both  s.i.;}
\\
\\
u+v-1+C(1-u,1-v),~\text{if  $\phi$ and $\psi$  are both s.d.;}
\\
\\
u-C(1-u,v),~\text{if  $\phi$ is s.d. and $\psi$ is s.i.;}
\\
\\
v-C(u,1-v),~\text{if  $\phi$ is s.i. and $\psi$  is s.d.}
\end{array}
\right.
\end{equation}

\begin{theorem}\label{th2.1}
Suppose $\phi(X)$ and $\psi(Y)$ are two strictly monotone (either increasing or decreasing) transformations of $X$ and $Y,$  respectively. Then,  

 \begin{equation*}
J_c(\phi(X),\psi(Y))=\left\{
	\begin{array}{ll}
	\displaystyle J_c(X,Y),~\text{if  $\phi$ and $\psi$  are both  s.i.;}
	\\
	\\
		\frac{1}{4}\int_{0}^{1}\int_{0}^{1} c^2(1-u,1-v)dudv,~\text{if  $\phi$ and $\psi$  are both s.d.;}
		\\
		\\
	\frac{1}{4}\int_{0}^{1}\int_{0}^{1} c^2(1-u,v)dudv,~\text{if  $\phi$ is s.d. and $\psi$ is s.i.;}
	\\
	\\
	\frac{1}{4}\int_{0}^{1}\int_{0}^{1} c^2(u,1-v)dudv,~\text{if  $\phi$ is s.i. and $\psi$  is s.d.,}
	\end{array}
	\right.
	\end{equation*}
where $c(1-u,1-v)=\frac{\partial^2 C(1-u,1-v)}{\partial u \partial v}$, $c(1-u,v)=-\frac{\partial^2 C(1-u,v)}{\partial u \partial v},$ and $c(u,1-v)=-\frac{\partial^2 C(u,1-v)}{\partial u \partial v}.$
\end{theorem}

\begin{proof}
 The proof of the theorem follows from Definition \ref{de2.1} and Eq. (\ref{eq2.4*}). Thus, it is omitted.
\end{proof}

\section{Cumulative copula extropy}
In the previous section, we have introduced density copula-based extropy. Here, we propose extropy based on distributional copula, and study its properties. The proposed measure is dubbed as cumulative copula extropy (CCEx). 
\begin{definition}\label{de3.1}
Let $C(\cdot,\cdot)$ be the copula function for a two-dimensional random vector $(X,Y)$. Then, the CCEx of $(X,Y)$ is defined as 
\begin{align}\label{eq3.1}
J_C(X,Y)=\frac{1}{4}\int_{0}^{1}\int_{0}^{1}C^2(u,v)dudv.
\end{align}
\end{definition}
Note that in a similar way, the CCEx can be extended for a general $k$-dimensional random vector. In this case, it is required to substitute a $k$-dimensional multivariate distributional copula in place of the bivariate copula in (\ref{eq3.1}) with a multiplying factor $(-\frac{1}{2})^{k}$ in place of $\frac{1}{4}$. In Table $2$, we present closed-form expressions of the CCEx of various copula functions, such as the product copula, extended FGM copula, iterated FGM copula, and Marshall and Olkin copula. Note that the FGM copula can be deduced from extended FGM copula by taking $p=1$. Further, the graphs of the CCEx of FGM copula, Marshall and Olkin copula (see p. 39, \cite{lai2009continuous}), and iterated FGM copula are provided in Figures $3(a)$, $3(b)$, and $3(c)$, respectively, implying that the CCEx is not monotone in general with respect to the parameters.

\begin{table}[h!]
	\caption {Cumulative copula extropy of some specific  copulas.}
	\centering 
	\scalebox{0.65}{\begin{tabular}{c c c c c c c c } 
			\hline\hline\vspace{.1cm} 
			Copula & Copula function $(C(u,v))$ & Cumulative copula extropy ($J_C(X,Y)$) \\
			\hline
			Product copula& $uv$ &$\frac{1}{16}$	\\	[2EX]
			Extended FGM copula & $[1+\theta(1-u)^p(1-v)^p]uv,~~|\theta|\le1,~p>0$ & $\frac{1}{36}+ \frac{\theta}{2(p+2)^2(p+3)^2}$\\
			&~&+$\frac{\theta^2}{4}\left[\frac{p^2}{(2p+1)(p+1)^2}+\frac{2(2p^2+6p+1)}{(p+1)(2p+1)(2p+3)^2}\right]$\\[1EX]
			{Iterated FGM copula} &	$[1+\alpha (1-u)(1-v)+\beta uv(1-u)(1-v)]uv, ~|\alpha|\le 1, |\beta|\le1$	& $\frac{1}{4}[\frac{1}{9}+\frac{\alpha}{72}+\frac{\beta}{200}+\frac{\alpha^2}{900}+\frac{241\alpha \beta}{1800}+\frac{\beta^2}{11025}]$\\[1EX]	
			{Marshall and Olkin copula} &  
			$uv\text{min}(u^{-\alpha},v^{-\beta})$=	$\begin{cases} 
			u^{1-\alpha}v,~u^{-\alpha}\le v^{-\beta},
			\\
			uv^{1-\beta},~ u^{-\alpha}\ge v^{-\beta}
			\end{cases}$
			$0\le \alpha,\beta\le 1$
			&
			$\begin{cases} 
			-\frac{1}{6(3-2\alpha)}\\
			-\frac{1}{6(3-2\beta)}
			\end{cases}$\\[1EX]			
			\hline
	\end{tabular}} 
	\label{tb1} 
\end{table}

\begin{figure}[h!]\label{fig2}
	\centering
	\subfigure[]{\label{c1}\includegraphics[height=1.26in]{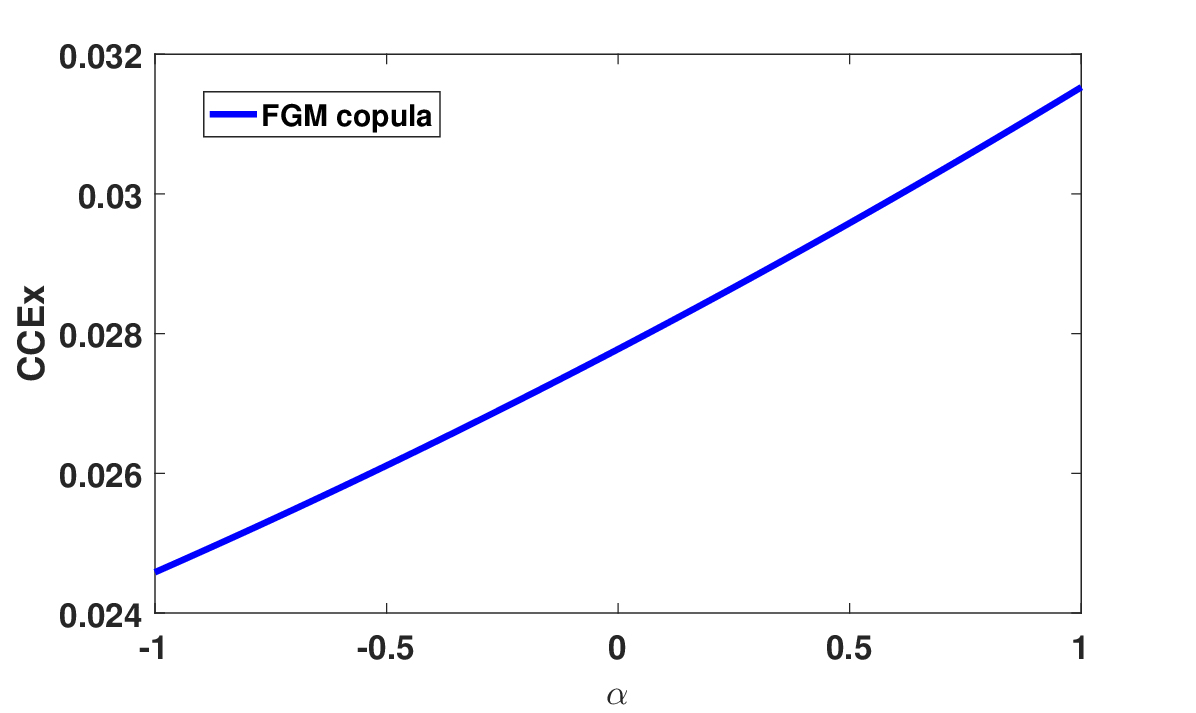}}
	\subfigure[]{\label{c1}\includegraphics[height=1.26in]{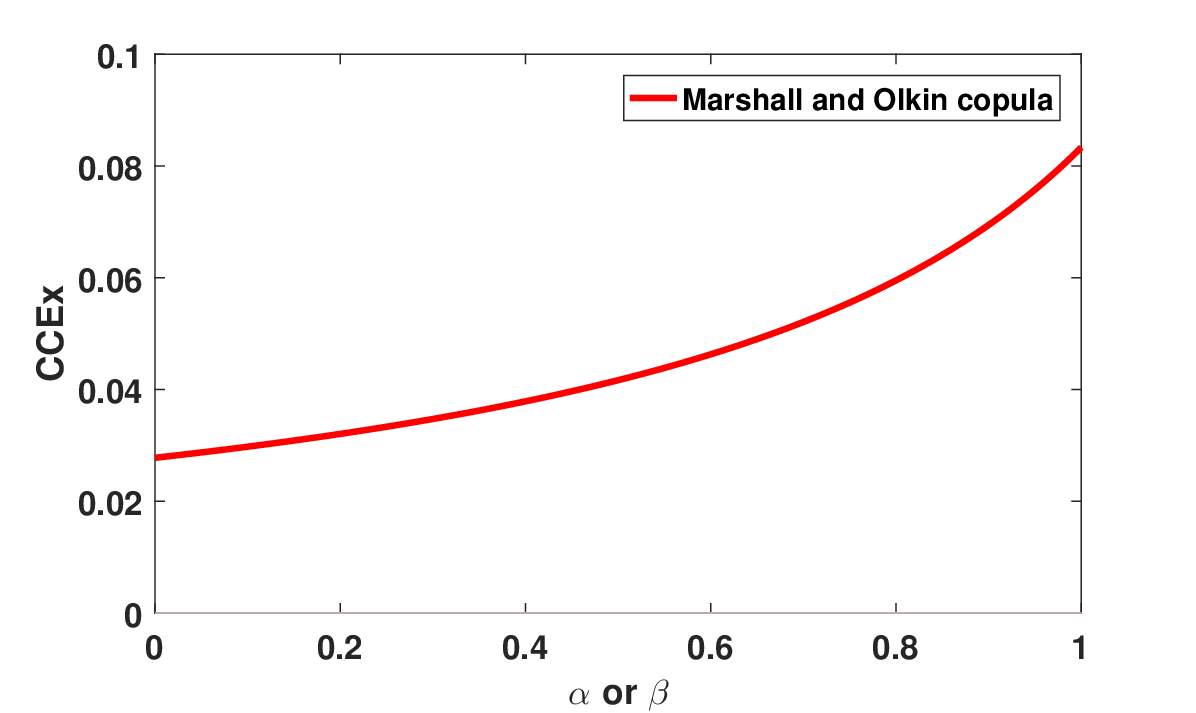}}
	\subfigure[]{\label{c1}\includegraphics[height=1.26in]{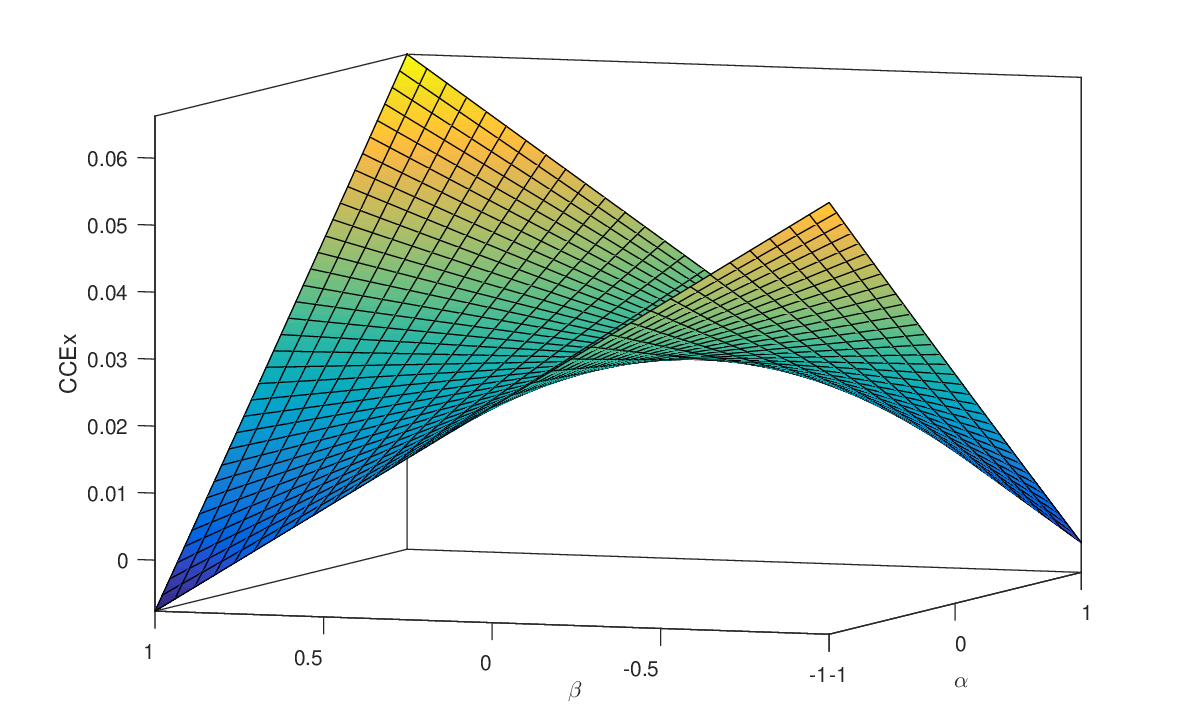}}
	\caption{Graphs of CCEx for $(a)$ FGM copula for $|\alpha|\le1$, $(b)$ Marshall and Olkin copula.  $(c)$ Surface plot of the CCEx for iterated FGM copula.}
\end{figure}

In various cases, the CCEx's are difficult to evaluate directly for some copulas.  Here, the method of transformation of random variables can be useful. Below, we show the effects of monotone transformations on CCEx.
\begin{theorem}\label{th3.1}
	Suppose $\phi(\cdot)$ and $\psi(\cdot)$ are two strictly monotone transformations of  $X$ and $Y$. Then, 
	
	\begin{equation*}
	J_C(\phi(X),\psi(Y))=\left\{
	\begin{array}{ll}
	\displaystyle J_C(X,Y),~\text{if  $\phi$ and $\psi$  are both  s.i.;}
	\\
	\\
	J_{\bar C}(X,Y),~\text{if  $\phi$ and $\psi$  are both s.d.;}
	\\
	\\
	\frac{1}{4}\int_{0}^{1}\int_{0}^{1} [u-C(u,v)]^2dudv,~\text{if  $\phi$ is s.d. and $\psi$ is s.i.;}
	\\
	\\
	\frac{1}{4}\int_{0}^{1}\int_{0}^{1} [v-C(u,v)]^2dudv,~\text{if  $\phi$ is s.i and $\psi$  is s.d.,}
	\end{array}
	\right.
	\end{equation*}
\end{theorem} 
where $J_{\bar C}(X,Y)$ is called the survival copula extropy, defined  in the next section.
\begin{proof}
	The proof of this theorem follows using similar arguments as in Theorem \ref{th2.1}, and thus it is omitted.
\end{proof}

We can also introduce another types of CCEx based on the sections of a copula (see p.11-12, \cite{nelsen2006introduction}). Mainly, there are three types of sections of  a copula. Let $C(\cdot,\cdot)$ be any copula and $a\in[0,1]$. 
 \begin{itemize}
 \item The copula $\mathcal{H}_C$ is called horizontal section  of $C(\cdot,\cdot)$ at $a$ if there exists a function from $[0,1]$ to $[0,1]$ given by $u\rightarrow C(u,a)$, i.e. $\mathcal{H}_C(u)=C(u,a)$;\
 \item The copula $\mathcal{V}_C$ is called vertical section  of $C(\cdot,\cdot)$ at $a$ if there exists a function from $[0,1]$ to $[0,1]$ given by $u\rightarrow C(a,u)$, i.e. $\mathcal{V}_C(u)=C(a,u)$;\
 \item The copula $\mathcal{D}_C$ is called diagonal section  of $C(\cdot,\cdot)$  if there exists a function from $[0,1]$ to $[0,1]$ given by $u\rightarrow C(u,u)$, i.e. $\mathcal{D}_C(u)=C(u,u)$.
 \end{itemize}
Now, based on the different sections of a copula, we have the following definition. 

 \begin{definition}\label{de3.2}
 Let $C(\cdot,\cdot)$ be a copula corresponding to $(X,Y)$. Then, for any number $a\in[0,1]$,
 \begin{itemize}
 \item [$(a)$] the cumulative horizontal copula extropy (denoted by $J_{\mathcal{H}_C}(u)$) is defined as
 \begin{align}
 J_{\mathcal{H}_C}(u)=\frac{1}{4}\int_{0}^{1}\mathcal{H}^2_C(u)du=\frac{1}{4}\int_{0}^{1}C^2(u,a)du;
 \end{align}
  \item [$(b)$] the cumulative vertical copula extropy (denoted by  $J_{\mathcal{V}_C}(u)$) is given as
  \begin{align}
  J_{\mathcal{V}_C}(u)=\frac{1}{4}\int_{0}^{1}\mathcal{V}^2_C(u)du=\frac{1}{4}\int_{0}^{1}C^2(a,u)du;
  \end{align}
  \item [$(c)$] the cumulative diagonal copula extropy (denoted by $J_{\mathcal{D}_C}(u)$) is defined as
    \begin{align}
    J_{\mathcal{D}_C}(u)=\frac{1}{4}\int_{0}^{1}\mathcal{D}^2_C(u)du=\frac{1}{4}\int_{0}^{1}C^2(u,u)du.
  \end{align}
 \end{itemize}
 \end{definition}
  
  In Table $3$, we obtain the closed-form expressions of $J_{\mathcal{H}_C}$, $J_{\mathcal{V}_C},$ and $J_{\mathcal{D}_C}$ for product and Cuadras-Auge family (see p. 15, \cite{nelsen2006introduction}) of  copulas.
\begin{table}[h!]
	\caption {Cumulative sections copula extropy for some copulas for any $a\in[0,1],$ and $\alpha\in[0,1]$.}
	\centering 
	\scalebox{0.9}{\begin{tabular}{c c c c c c c c } 
			\hline\hline\vspace{.1cm} 
			Copula & Copula function $(C(u,v))$ & $J_{\mathcal{H}_C}(u)$ & $J_{\mathcal{V}_C}(u)$ & $J_{\mathcal{D}_C}(u)$ \\
			\hline
			Product copula & $uv$ & $\frac{a^2}{12}$ & $\frac{a^2}{12}$ &$\frac{1}{20}$\\
			 Cuadras-Auge family&
			 
			$[\text{min}(u,v)]^\alpha[uv]^{1-\alpha}$=
			$\begin{cases}
			uv^{1-\alpha}, &~u\le v
			\\
			u^{1-\alpha}v, &~ v\le u\\
			\end{cases}$
			&
		$	\begin{cases}
			 \frac{a^{2-2\alpha}}{12}
				\\
			\frac{a^2}{4(3-2\alpha)}
			\end{cases}$
		 & 
		 $\begin{cases}
			 \frac{a^2}{4(3-2\alpha)}
			\\
        \frac{a^{2-2\alpha}}{12}
		\end{cases}$
				& $\frac{1}{4(3-\alpha)}$\\

					\hline
		\end{tabular}} 
\end{table}

We recall that the concept of radial symmetry is a generalization of  the notion of univariate symmetry. A two-dimensional random vector $(X_1,X_2)$ is said to be radially symmetric about $(a_1,a_2)$ if the random vector $(X_{1}-a_1,X_2-a_2)$ has the same distribution as the random vector $(a_1-X_1,a_2-X_2).$ The famous examples of radially symmetric copulas in two-dimension are Frank and FGM families. 
\begin{remark}
 Note that $J_{\mathcal{H}_C}(u)$ and $J_{\mathcal{V}_C}(u)$ are equal when random vector $(X,Y)$ is radially symmetric. For example, these two extropies of the product copula are same (see Table $3$).  
\end{remark}

In the following, we establish a bound of the cumulative diagonal copula extropy.
\begin{proposition}\label{prop3.1}
For a copula function $C(\cdot,\cdot)$, we get
\begin{align}\label{eq3.5}
J_{\mathcal{D}_C}(u)\ge\frac{1}{12}.
\end{align}
\end{proposition}
\begin{proof}
From Lemma $4.1$ in \cite{capaldo2023generalized}, we have
\begin{equation}\label{eq3.6}
C(u,u)\ge2u-1, ~~\text{for all $u\in[0,1]$.}
\end{equation}
Using (\ref{eq3.6}) in Definition \ref{de3.2}$(c)$, the result in (\ref{eq3.5}) easily follows, completing the proof.
\end{proof}

In order to justify the result in Proposition \ref{prop3.1}, we plot the cumulative diagonal copula extropy for Cuadras-Auge family in Figure $3(a),$ which is clearly larger than $\frac{1}{12}$.

\begin{figure}[h!]\label{fig2}
	\centering
	\subfigure[]{\label{c1}\includegraphics[height=1.92in]{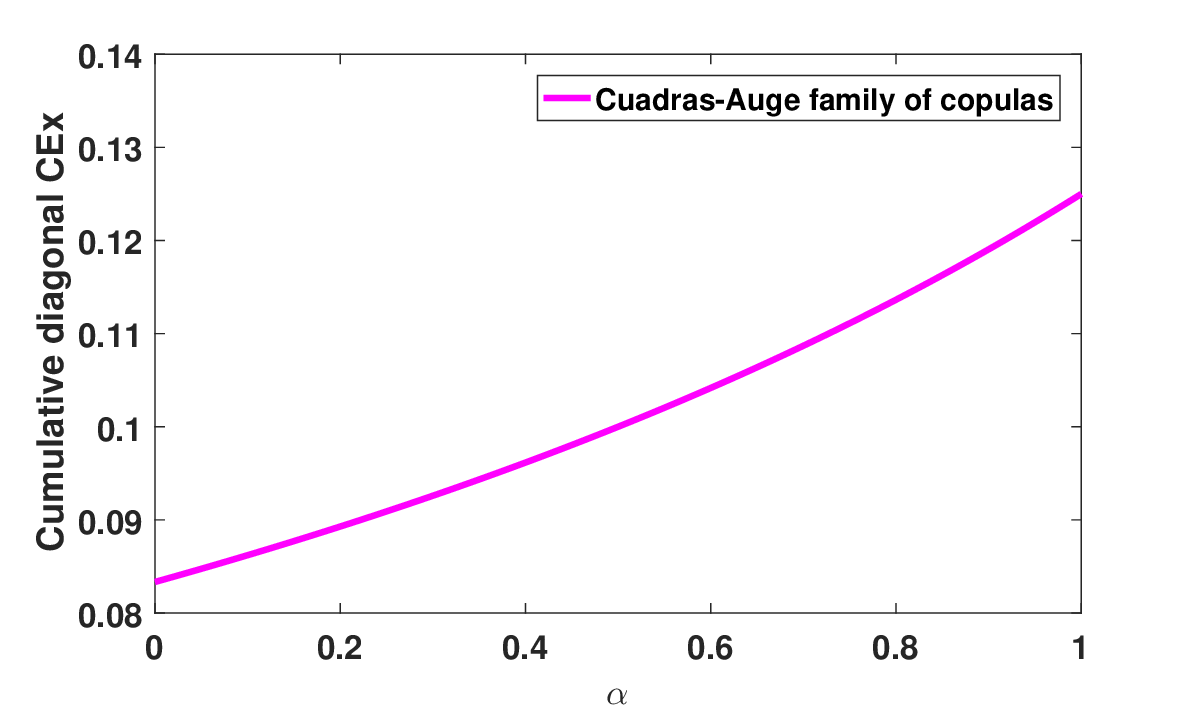}}
	\subfigure[]{\label{c1}\includegraphics[height=1.92in]{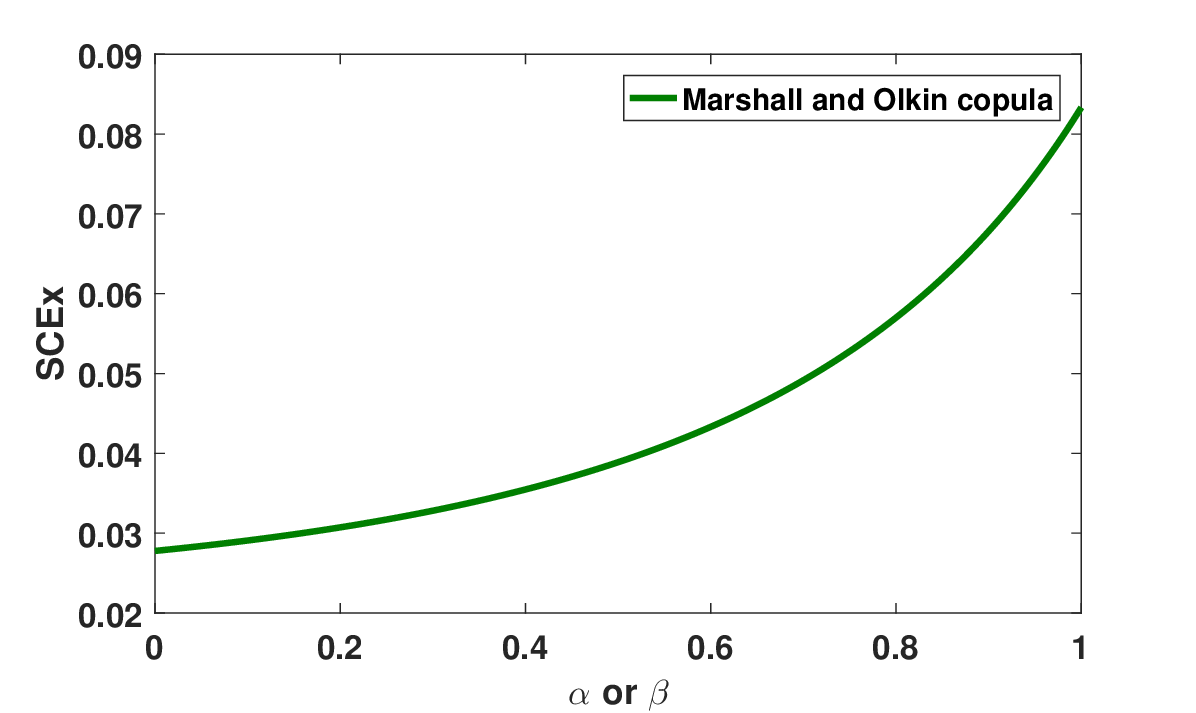}}
	\caption{$(a)$ Graph for cumulative diagonal copula extropy  of  Cuadras-Auge family of copulas for  $\alpha\in[0,1]$. $(b)$ Graph for SCEx of Marshall and Olkin  family of copulas in Table $6$}
\end{figure}

%
Suppose $X$ and $Y$ are two random variables and their corresponding copula is $C(\cdot,\cdot)$. Then, the dual of the copula $C(\cdot,\cdot)$ is defined as 
\begin{align}\label{eq3.7*}
C^*(u,v)=u+v-C(u,v),~~\text{for all $u,v\in[0,1]$}.
\end{align} 
Note that the dual of the copula in (\ref{eq3.7*}) is not a copula. However, it is closely related to a copula. The expression is a probability of an event involving $X$ and $Y$. For more details about the dual copula, we refer to \cite{nelsen2006introduction}, (p. $33$). The  co-copula and  dual copula functions can be used to obtain  inaccuracy measure in information theory. We introduce a new information measure in the following.
\begin{definition}
Let the copula and dual of the copula corresponding to $(X,Y)$ be denoted by $C(\cdot,\cdot)$ and $C^*(\cdot,\cdot),$ respectively. Then, the dual copula extropy is defined as 
\begin{align}\label{eq3.8}
J_{C^*}(X,Y)=\frac{1}{4}\int_{0}^{1}\int_{0}^{1}C{^*}^2(u,v)dudv,~~\text{for all $u,v\in[0,1]$}.
\end{align}
\end{definition}

The expressions of the dual copula extropy are given for different copulas in Table $4$.  
  
\begin{table}[h!]
	\caption {Dual copula extropy of some specific copulas.}
	\centering 
	\scalebox{0.8}{\begin{tabular}{c c c c c c c c } 
			\hline\hline\vspace{.1cm} 
			Copula & Copula function $(C(u,v))$ & Dual copula extropy ($J_C^*(X,Y)$) \\
			\hline
				Product copula& $uv$ &$\frac{11}{72}$	\\	[2EX]
	Marshall and Olkin copula &  
    $uv\text{min}(u^{-\alpha},v^{-\beta})$=	$\begin{cases} 
				u^{1-\alpha}v,~u^{-\alpha}\le v^{-\beta}
				\\
				uv^{1-\beta},~ u^{-\alpha}\ge v^{-\beta}
			\end{cases}$
				&
			$\begin{cases} 
		\frac{1}{4}\left[\frac{7}{6}-\frac{12-5\alpha}{3(2-\alpha)(3-\alpha)}-\frac{1}{3(3-2\alpha)}\right]\\
	\frac{1}{4}\left[\frac{7}{6}-\frac{12-5\beta}{3(2-\beta)(3-\beta)}-\frac{1}{3(3-2\beta)}\right]
	\end{cases}$\\				
			
	\hline	 		
	\end{tabular}} 
	\label{tb1} 
\end{table}

Next, we obtain a relation  between CCEx and dual copula extropy.
\begin{proposition}\label{prop3.2}
Let $(X,Y)$ be a two-dimensional random vector with copula $C(\cdot,\cdot)$. Further, assume that the dual of the copula is $C^*(\cdot,\cdot).$ Then, 
\begin{align}
J_C(X,Y)=J_{C^*}(X,Y)+\frac{1}{2}\mathcal{R}(X,Y)-\frac{7}{24},
\end{align} 
\end{proposition}
where $J_C(X,Y)$ and $J_{C^*}(X,Y)$ are given by (\ref{eq3.1}) and (\ref{eq3.8}), respectively, and
\begin{eqnarray}\label{eq3.10}
\mathcal{R}(X,Y)=\int_{0}^{1}\int_{0}^{1}(u+v) C(u,v)dudv.
\end{eqnarray}

\begin{proof}
Using (\ref{eq3.7*}) in (\ref{eq3.8}), we obtain
\begin{align*}
J_{{C^*}}(X,Y)=\frac{1}{4}\int_{0}^{1}\int_{0}^{1}{C^*}^2(u,v)dudv
&=\frac{1}{4}\int_{0}^{1}\int_{0}^{1}[u+v- C(u,v)]^2dudv\\
&=J_{{C}}(X,Y)+\frac{7}{24}-\frac{1}{2}\int_{0}^{1}\int_{0}^{1}(u+v) C(u,v)dudv,
\end{align*}
following the desired output. 
\end{proof}

The following example validates Proposition \ref{prop3.2}.

\begin{example}
		Consider Cauadras-Auge family of copulas in Table $3$. From (\ref{eq3.10}), we obtain
		\begin{align}\label{eq3.11}
		\mathcal{R}(X,Y)=\frac{12-5\alpha}{6(2-\alpha)(3-\alpha)}.
		\end{align}
		Further, 
			\begin{align}\label{eq3.12}
		J_{C^*}(X,Y)=\frac{1}{4}\bigg[\frac{7}{6}-\frac{12-5\alpha}{3(2-\alpha)(3-\alpha)}+\frac{1}{3(3-2\alpha)}\bigg]~\mbox{and}~J_C(X,Y)=\frac{1}{12(3-2\theta)}.
		\end{align}
		Now, from (\ref{eq3.11}) and (\ref{eq3.12}) we obtain
		\begin{align*}
		J_{C^*}(X,Y)+\frac{1}{2}\mathcal{R}(X,Y)-\frac{7}{24}
		=\frac{1}{12(3-2\theta)}
		=J_C(X,Y),
		\end{align*}
		validating the result in Proposition \ref{prop3.2}.
\end{example}

 Next, we will establish connection between concordance ordering and ordering between CCEx's. First, we present the definition of the concordance ordering.

\begin{definition}\label{de3.3} (\cite{nelsen2006introduction}, p. 39)
Suppose $C_1(\cdot,\cdot)$ and $C_2(\cdot,\cdot)$ are two copulas. Then, $C_1(\cdot,\cdot)$ is smaller (larger) than $C_2(\cdot,\cdot)$ in concordance order, denoted by $C_1\prec (\succ)C_2,$ if $C_1(u,v)\le (\ge)C_2(u,v),$ for all $u,v\in[0,1]$. 
\end{definition}

We recall that $C_1(\cdot,\cdot)$ is smaller than $C_2(\cdot,\cdot)$ in concordance order means that $C_1(\cdot,\cdot)$ is less positively quadrant dependent than $C_2(\cdot,\cdot).$  Further, it can be established that all beta families of generators, and many alpha families, generate families of copulas which satisfy concordance order.  It is well known that $$C_1\prec (\succ)C_2\Leftrightarrow \bar C_1\prec (\succ)\bar C_2.$$ 

The following proposition states that less positively quadrant dependent copula will produce a system with less uncertainty (risk). 
 
 \begin{proposition}\label{pro3.3}
 Suppose $C_1(\cdot,\cdot)$ and $C_2(\cdot,\cdot)$ are two copulas with  CCEx's $J_{C_1}(\cdot,\cdot)$ and $J_{C_2}(\cdot,\cdot)$, respectively. Then, 
 \begin{align}\label{eq3.7}
 C_1\prec(\succ)C_2\Longrightarrow J_{C_1}\le (\ge)J_{C_2}.
 \end{align}
 \end{proposition}\label{pro3.2}
 \begin{proof}
 Using Definitions \ref{de3.1} and \ref{de3.3}, the result in (\ref{eq3.7}) easily follows. 
 \end{proof}
   
Next, we consider an example to illustrate the result in Proposition \ref{pro3.3}. 
\begin{example}
		Consider Cuadras-Auge family of copulas provided in Table $3$. For  $\alpha,\alpha_1\in[0,1]$ and $\alpha\ge \alpha_1$, it can be shown that $C^\alpha(u,v)\ge C^{\alpha_1}(u,v)$, that is,  $C^\alpha\succ C^{\alpha_1}.$ Thus, from Proposition \ref{pro3.3}, we obtain $J_{C^{\alpha}}\ge J_{C^{\alpha_{1}}}$, which can be easily noticed from Table $5$.		
		\begin{table}[h!]
			\caption {The values of cumulative copula extropy of Cuadras-Auge family of copulas for different values of $\alpha$ in $[0,1]$.}
			\centering 
			\scalebox{.95}{\begin{tabular}{c c c c c c c c c c } 
					\hline\hline\vspace{.1cm} 
						$\alpha$ & 0.1& 0.2 & 0.3& 0.4& 0.5& 0.6& 0.7& 0.8& 0.9 \\
					\hline
					$J_C{^{\alpha}}$ & 0.02976 & 0.03205 & 0.03472& 0.03789& 0.04167& 0.04630& 0.005208& 0.05952& 0.06944	\\	[1EX]
					\hline
			\end{tabular}} 
			\label{tb1} 
		\end{table}
\end{example}   
 
The converse part of Proposition \ref{pro3.3} is not true as seen from the following example.
 \begin{example}
 For a copula $C_1(u,v)=uv$ and $C_2(u,v)=uv[1+(1-u)(1-v)]$, the values of cumulative copula extropies are $J_{C_1}=\frac{1}{16}=0.0625$ and $J_{C_2}=0.0312499$. It is clear that $J_{C_1}\ge J_{C_2}$ but $C_2(u,v)\ge C_1(u,v)$ i.e. $(C_1\prec C_2)$.
 \end{example}

In the next proposition, we observe that the converse of Proposition \ref{pro3.3} is also true with different assumptions. 
 \begin{proposition}\label{pro3.4}
 Suppose $\mathcal{F}$ is the class of copulas which are concordance ordered. For any two copulas $C_1$ and $C_2\in \mathcal{F},$ and for all $u,v\in[0,1]$, we have
 \begin{align}
 J_{C_1}\le (\ge)J_{C_2}\Longrightarrow C_1(u,v)\le (\ge) C_2(u,v).
 \end{align}
 \end{proposition}
 \begin{proof}
 Let $J_{C_1}\le (\ge) J_{C_2}$. Then, from the definition of CCEx, we obtain
 \begin{align*}
 \frac{1}{4}\int_{0}^{1}\int_{0}^{1}C^2_{1}(u,v)dudv \le (\ge)\frac{1}{4}\int_{0}^{1}\int_{0}^{1}C^2_{2}(u,v)dudv,
 \end{align*}
 implying $C_1(u,v)\le (\ge)C_2(u,v),$ as desired. 
 \end{proof}

\section{Survival copula extropy}
The preceding section deals with the information measure depending on the distributional copula. Here, we introduce extropy based on survival copula, known as the survival copula extropy (SCEx). First we provide the following definition. 
\begin{definition}\label{de4.1}
Suppose $(X,Y)$ is a two-dimensional random vector with survival copula $\bar C(\cdot,\cdot)$. Then, the SCEx is defined as 
\begin{align}\label{eq4.1}
J_{\bar C}(X,Y)=\frac{1}{4}\int_{0}^{1}\int_{0}^{1}\bar C^2(u,v)dudv.
\end{align} 
\end{definition} 

We obtain the closed-form expressions of the SCEx for some copula functions, which are provided in  Table $6$. Further, we plot the SCEx of Marshall and Olkin copula in Figure $3(b)$. 
\begin{table}[h!]
	\caption {The closed-form expressions of the SCEx for some specific  copulas.}
	\centering 
	\scalebox{0.68  }{\begin{tabular}{c c c c c c c c } 
			\hline\hline\vspace{.1cm} 
			Copula & Copula function $(C(u,v))$ & Survival copula extropy ($J_{\bar C}(X,Y)$) \\
			\hline
			Product copula& $uv$ &$\frac{1}{16}$	\\	[2EX]
			FGM copula & $uv[1+\theta (1-u)(1-v)], ~~|\theta|\le1$  & $\frac{1}{4}(\frac{1}{9}+\frac{\theta}{72}+\frac{\theta^2}{900})$\\[2EX]

%
			{ Marshall and Olkin copula} &
			
			$uv\text{min}\{u^{-\alpha},v^{-\beta}\}$=	
			$\begin{cases}
			u^{1-\alpha}v,&~u^{-\alpha}\le v^{-\beta}
			\\
			uv^{1-\beta},&~ u^{-\alpha}\ge v^{-\beta}
			\end{cases}$
			&
			$\begin{cases}
			\frac{1}{12}[\frac{1}{2}+\frac{2\alpha-3}{(2-\alpha)(3-\alpha)}+\frac{1}{(3-2\alpha)}]
			\\
			\frac{1}{12}[\frac{1}{2}+\frac{2\beta-3}{(2-\beta)(3-\beta)}+\frac{1}{(3-2\beta)}]
			\end{cases}$\\	[2EX]
			{Shih and Louis's copula} &

			$\begin{cases}
			(1-\rho)uv+\rho\text{min}\{u,v\},~~\rho>0
			\\
			(1+\rho)uv+\rho(u+v-1)\xi(u+v-1),\\
			\rho\le0,~\mbox{where}~ \xi(a)=1, if a\ge0 ~\text{and} ~	\xi(a)=0, if a<0				
			\end{cases}$
			&
			$\begin{cases}
			\frac{1}{36}[2+5\rho-4\rho^2],~\rho>0
			\\
			\frac{1}{72}[2+7\rho+2\rho^2],~\rho\le0~ \text{and}~\xi(a)=1\\
			\frac{1}{72}[4-22\rho-5\rho^2],~\rho\le0~ \text{and}~\xi(a)=0
			\end{cases}$\\	[2EX]
			
			\hline
	\end{tabular}} 
	\label{tb1} 
\end{table}

Next, we consider monotone transformations of random variables to see their effect on SCEx. The following result, which can be proved similar to Theorem $2.4.3$ and Theorem $2.4.4$ of \cite{nelsen2006introduction}, is useful here:

 \begin{equation}\label{eq2.5}
\bar{C}_{\phi(X)\psi(Y)}(u,v)=\left\{
\begin{array}{ll}
\bar{C}(u,v),~\text{if  $\phi$ and $\psi$  are both  s.i.;}
\\

C(u,v),~\text{if  $\phi$ and $\psi$  are both s.d.;}
\\

u-\bar{C}(u,1-v),~\text{if  $\phi$ is s.d. and $\psi$ is s.i.;}
\\

v-\bar{C}(1-u,v),~\text{if  $\phi$ is s.i. and $\psi$  is s.d.}
\end{array}
\right.
\end{equation}

\begin{theorem}
	Suppose $\phi(\cdot)$ and $\psi(\cdot)$ are two strictly monotone transformations for $X$ and $Y$, respectively. Then, 
	
	\begin{equation*}
	J_{\bar C}(\phi(X),\psi(Y))=\left\{
	\begin{array}{ll}
	\displaystyle J_{\bar C}(X,Y),~\text{if  $\phi$ and $\psi$  are both  s.i.;}
	\\
	J_{ C}(X,Y),~\text{if  $\phi$ and $\psi$  are both s.d.;}
	\\
	
	\frac{1}{4}\int_{0}^{1}\int_{0}^{1} [u-\bar C(u,1-v)]^2dudv,~\text{if  $\phi$ is s.d and $\psi$ is s.i;}
	\\
	
	\frac{1}{4}\int_{0}^{1}\int_{0}^{1} [v-\bar C(1-u,v)]^2dudv,~\text{if  $\phi$ is s.i and $\psi$  is s.d;}
	\end{array}
	\right.
	\end{equation*}
	where $J_{ C}(X,Y)$ is  defined in (\ref{eq3.1}).
\end{theorem} 

\begin{proof}
	The proof  is similar to that of Theorem \ref{th2.1}, and thus it is not presented here.
\end{proof}

\begin{remark}
 Generally, the SCEx and CCEx are different for a specific copula, i.e., $$J_C(X,Y)\neq J_{\bar C}(X,Y).$$ However, for a radially symmetric random vector $(X,Y)$ the information measures  SCEx and CCEx are equal, which can be observed for the case of FGM family of copulas.
\end{remark}

Suppose $X$ and $Y$ are two random variables with marginals $F(\cdot)$ and $G(\cdot),$ respectively. Then, the co-copula is expressed as (see \cite{nelsen2006introduction}, p. $33$)
\begin{align}\label{eq4.2}
\bar{C^*}(F(x),G(y))=P[X>x~\text{or}~ Y>y]
\Rightarrow \bar{C^*}(u,v)=1-C(1-u,1-v),
\end{align}
for all $u,v\in[0,1]$.  The probabilistic quantity $P[X>x ~\text{or}~ Y>y]$ is related to  survival copula function has importance in reliability theory and survival analysis (see  \cite{hosseini2021discussion}).
Analogous to the SCEx, the co-copula extropy  of $(X,Y)$ is defined as 
\begin{align}\label{eq4.3}
J_{\bar{C^*}}(X,Y)=\frac{1}{4}\int_{0}^{1}\int_{0}^{1}\bar{C^*}^2(u,v)dudv.
\end{align}
  
In Table $7$, we obtain closed form expressions of the co-copula extropy for different copula functions.  
\begin{table}[h!]
	\caption {The expressions of  co-copula extropy for some specific  copulas.}
	\centering 
	\scalebox{0.87}{\begin{tabular}{c c c c c c c c } 
			\hline\hline\vspace{.1cm} 
			Copula & Copula function $(C(u,v))$ & Co-copula extropy ($J_{\bar C}^*(X,Y)$) \\
			\hline
				Product copula& $uv$ &$\frac{11}{72}$	\\	[2EX]
	Marshall and Olkin copula &  $uv\text{min}(u^{-\alpha},v^{-\beta})$=
			$\begin{cases} 
				u^{1-\alpha}v, &~u^{-\alpha}\le v^{-\beta}\\
				uv^{1-\beta},&~ u^{-\alpha}\ge v^{-\beta}
        	\end{cases}$
				&
				$\begin{cases}
				 \frac{1}{4}[1-\frac{1}{(2-\theta)}+\frac{1}{3(3-2\theta)}]\\
				\frac{1}{4}[1-\frac{1}{(2-\beta)}+\frac{1}{3(3-2\beta)}]
			 \end{cases}$\\
	\hline	 		
	\end{tabular}} 
	\label{tb1} 
\end{table}

We remark that the co-copula function is closely related to survival copula function. Next, we establish a relation between SCEx and co-copula extropy.
\begin{proposition}\label{prop4.1}
Suppose $(X,Y)$ is a  two-dimensional random vector with co-copula and survival copula  $\bar{C^*}(\cdot,\cdot)$ and $\bar C(\cdot,\cdot),$ respectively. Then, 
\begin{align}\label{eq4.4}
J_{\bar C}(X,Y)=J_{\bar{C^*}}(X,Y)+\frac{1}{2}\mathcal{R}^*(X,Y)-\frac{7}{24},
\end{align}
where 
\begin{eqnarray}\label{eq4.5*}
\mathcal{R}^*(X,Y)=\int_{0}^{1}\int_{0}^{1}(u+v)\bar C(u,v)dudv.
\end{eqnarray}
\end{proposition}

\begin{proof}
Employing the relation
$\bar{C^*}(u,v)=u+v-\bar C(u,v),$ the proof follows using analogous arguments as in  Proposition \ref{prop3.2}.
\end{proof}

Next, we consider an example to illustrate Proposition \ref{prop4.1}.

\begin{example}
	For FGM copula (see Table $6$), we obtain
	\begin{align}\label{eq4.6}
	\mathcal{R}^*(X,Y)=\frac{1}{36}(\theta+5)
	\end{align}
	and 
		\begin{align}\label{eq4.7}
	J_{\bar C}(X,Y)=\frac{1}{4}\left(\frac{1}{9}+\frac{\theta}{72}+\frac{\theta^2}{900}\right).
	\end{align}
	Now, from (\ref{eq4.6}) and (\ref{eq4.7})
	\begin{align*}
	J_{\bar C}(X,Y)-\frac{1}{2}\mathcal{R}^*(X,Y)+\frac{7}{24}&=\frac{1}{4}\left(\frac{1}{9}+\frac{\theta}{72}+\frac{\theta^2}{900}\right)-\frac{1}{72}(\theta+5)+\frac{7}{24}
	=J_{\bar{C^*}}(X,Y).
	\end{align*}
\end{example}

Different copulas can capture varying degrees and types of dependence. Concordance order enables the comparison of these structures. Now, we show the effect of SCEx for concordance ordering between two survival copulas.

\begin{proposition}\label{prop4.2}
For two survival copulas $\bar C_1(\cdot,\cdot)$ and $\bar C_2(\cdot,\cdot)$, we have
\begin{align}
 \bar C_1\prec(\succ)\bar C_2\Longrightarrow J_{\bar C_1}\le (\ge)J_{\bar C_2}.
\end{align}
\end{proposition}
\begin{proof}
The proof is straightforward, and thus it is omitted. 
\end{proof}

{Next, we consider an example for illustration of the result in Proposition \ref{prop4.2}.  
	\begin{example}
		For FGM copula, it is easy to check that $\bar C_{\theta_1}\le \bar C_{\theta_2},$ for $\theta_1\le \theta_2$. Further,
		\begin{align*}
		J_{\bar C_{\theta}}(X,Y)=\frac{1}{4}\left(\frac{1}{9}+\frac{\theta}{72}+\frac{\theta^2}{900}\right).
		\end{align*}
		Thus, clearly $J_{C_{\theta_1}}\le J_{C_{\theta_2}}$ holds, for $\theta_1\le\theta_2$ (see Table $8$).
			\begin{table}[h!]
			\caption {The numerical values of the survival copula extropy for FGM copula.}
			\centering 
			\scalebox{.8}{\begin{tabular}{c c c c c c c c c } 
					\hline\hline\vspace{.1cm} 
					$\theta$ & $0.1$ & $0.2$ & $0.3$& $0.4$& $0.5$ & $0.6$ &$0.7$ & $0.8$ \\
					\hline\vspace{.1cm}
					$J_{\bar C}$ &$0.02813$ &$0.02848$ &$0.02884$& $0.02921$&$0.02958$&$0.02996$&$0.03034$& $0.03073$\\[1EX]
					\hline
			\end{tabular}} 
		\end{table}
\end{example}}
The converse of Theorem \ref{prop4.2} is not true (see Example $3.3$ which is same for SCEx). In this regard, we propose a theorem with an additional condition.
\begin{proposition}\label{prop4.3}
 Suppose $\mathcal{G}$ is the class of copulas which are concordance ordered. For any two survival copula $\bar C_1$ and $\bar C_2\in \mathcal{G}$ and for all $u,v\in[0,1]$,
 \begin{align}
 J_{\bar C_1}\le (\ge)J_{\bar C_2}\Longrightarrow \bar C_1(u,v)\le (\ge) \bar C_2(u,v).
 \end{align}
 \end{proposition}
\begin{proof}
The proof is similar to the proof of Proposition \ref{pro3.4}. Hence, it is skipped.
\end{proof}

Now, we consider an example for illustrating Proposition \ref{prop4.3}.
\begin{example}
	For Linear Spearman copula (see \cite{joe1997multivariate}, p. $148$), given by $C_\alpha(u,v)=(1-\alpha)uv+\alpha\min\{u,v\} $,~$0\le \alpha\le 1$, we obtain
	\begin{align}\label{eq4.8}
	J_{\bar C_\alpha}(X,Y)=\frac{1}{36}(\alpha^2+\alpha+1).
	\end{align}
	It is clear from (\ref{eq4.8}) that $J_{\bar C_\alpha}\ge(\le)J_{\bar C_\beta},$ for $\alpha\ge(\le)\beta$.
	Further, 
	\begin{equation}\label{eq4.11}
	\bar C_\alpha(u,v)=\left\{
	\begin{array}{ll}
	\displaystyle (1-\alpha)uv+\alpha v,~~u<v
	\\
	(1-\alpha)uv+\alpha u,~~u>v
	\end{array}
	\right.
	\end{equation}
	From (\ref{eq4.11}), for $\alpha\ge(\le)\beta$ we obtain $\bar C_\alpha\ge(\le)\bar C_\beta$ since 
	\begin{align*}
	\bar C_\alpha(u,v)-\bar C_\beta(u,v)&=\{(1-\alpha)uv+\alpha v\}-\{(1-\beta)uv+\beta v\}\\
	&=(\beta-\alpha)uv+(\alpha-\beta)v\\
	&=(\alpha-\beta)(1-u)v\ge(\le)0.
	\end{align*}
	
\end{example}

\section{Dependence measures}
Here, we will see how the proposed measures are related to some existing dependence measures between $X$ and $Y.$  Suppose $X$ and $Y$ have respective CDFs $F(\cdot)$ and $G(\cdot)$; and joint distribution function $H(\cdot,\cdot)$ with corresponding copula $C(\cdot,\cdot)$ and survival copula $\bar{C}(\cdot,\cdot)$. Then, $X$ and $Y$ are positively (negatively) quadrant dependent PQD (NQD), if $H(x,y)\ge (\le) F(x)G(y)$ for all $(x,y)\in\mathbb{R}^2$, equivalently
\begin{align}\label{eq5.1}
C(u,v)~or~\bar{C}(u,v)\ge (\le) uv,~~\text{for all $u,v\in[0,1]$.}
\end{align}  
There are several copulas satisfying PQD and NQD properties. For example, FGM, Nelsen's, Marshall and Olkin, and Cuadras-Auge families of copulas  are PQD, whereas Shih and Louis's copula for $\rho\le0$ (see Table $6$)  is NQD. In the following proposition, we obtain bound for CCEx when a bivariate random vector is PQD and NQD.
\begin{proposition}\label{prop5.1}
	For  a PQD (NQD) random vector $(X,Y)$, we obtain
	\begin{align*}
	J_C(X,Y)\ge(\le)\frac{1}{36}.
	\end{align*}
\end{proposition}
\begin{proof}
	Under the assumptions made, from (\ref{eq5.1}), we obtain $C^2(u,v)\ge(\le)u^2v^2$. Thus,
	\begin{align*}
	\frac{1}{4}\int_{0}^{1}\int_{0}^{1}C^2(u,v)dudv\ge(\le)\frac{1}{4}\int_{0}^{1}\int_{0}^{1}u^2v^2dudv
	\Rightarrow J_C(X,Y)\ge(\le)\frac{1}{36},
	\end{align*}
	as desired.
\end{proof}

\begin{remark}\label{rem5.1}
	We note that bound for the SCEx can be obtained similarly as in Proposition \ref{prop5.1} when the random vector is PQD (NQD).
\end{remark}


In the next result, we establish conditions under which the random variables are more PQD.
\begin{proposition}\label{pro5.3}
	Suppose $C_1,C_2\in \mathcal{F}$. Then, we have $J_{C_1}\ge J_{C_2}$ if and only if $C_1$ is more PQD $(\gg)$ than $C_2$.
\end{proposition}

\begin{proof}
	It is well known that $C_1$ is more PQD than $C_2$ if and only if $C_{1}(u,v)\ge C_{2}(u,v)$ for all $(u,v)\in[0,1]\times[0,1]$ (see \cite{joe1997multivariate}). Thus,
	$$
	 C^2_1(u,v)\ge C^2_2(u,v)
	\Leftrightarrow \frac{1}{4}\int_{0}^{1}\int_{0}^{1}C^2_1(u,v)dudv\ge \frac{1}{4}\int_{0}^{1}\int_{0}^{1}C^2_2(u,v)dudv\Leftrightarrow J_{C_1}\ge J_{C_2},
	$$
	as desired.
\end{proof}

Next, we observe that a similar result for SCEx also holds. The proof is omitted here due to its simplicity. 

\begin{proposition}\label{prop5.4}
	Suppose two survival copulas $\bar C_1,\bar C_2\in \mathcal{G}$. Then,
$J_{\bar C_1}\ge J_{\bar C_2}$ if and only if $\bar C_1$ is more PQD than $\bar C_2$.
\end{proposition}

The following example validates Proposition \ref{pro5.3}.

\begin{example}\label{ex5.1}
		Consider the extended FGM family of copulas, given in Table $2$. Here, $C_1\gg C_p,$ for $p\ge1$ and $0\le \theta\le 1$. We obtain 
		\begin{align*}
		J_{C_p}=\frac{1}{4}\left[\frac{1}{9}+8\theta\left\{\frac{1}{(2+p)^2}-\frac{p^2+4p+2}{(1+p)^2(3+p)^2}\right\}+\theta^2\left\{\frac{1}{(1+p)^2}-\frac{8(2p^2+4p+1)}{(1+2p)^2(3+2p)^2}\right\}\right].
		\end{align*}
		From Figure $4(a)$, it is clear that $J_{C_1}\ge J_{C_p}$ for $p=2$, validating the desired result.
\end{example}

\begin{figure}[h!]\label{fig2}
	\centering
	\subfigure[]{\label{c1}\includegraphics[height=1.5in]{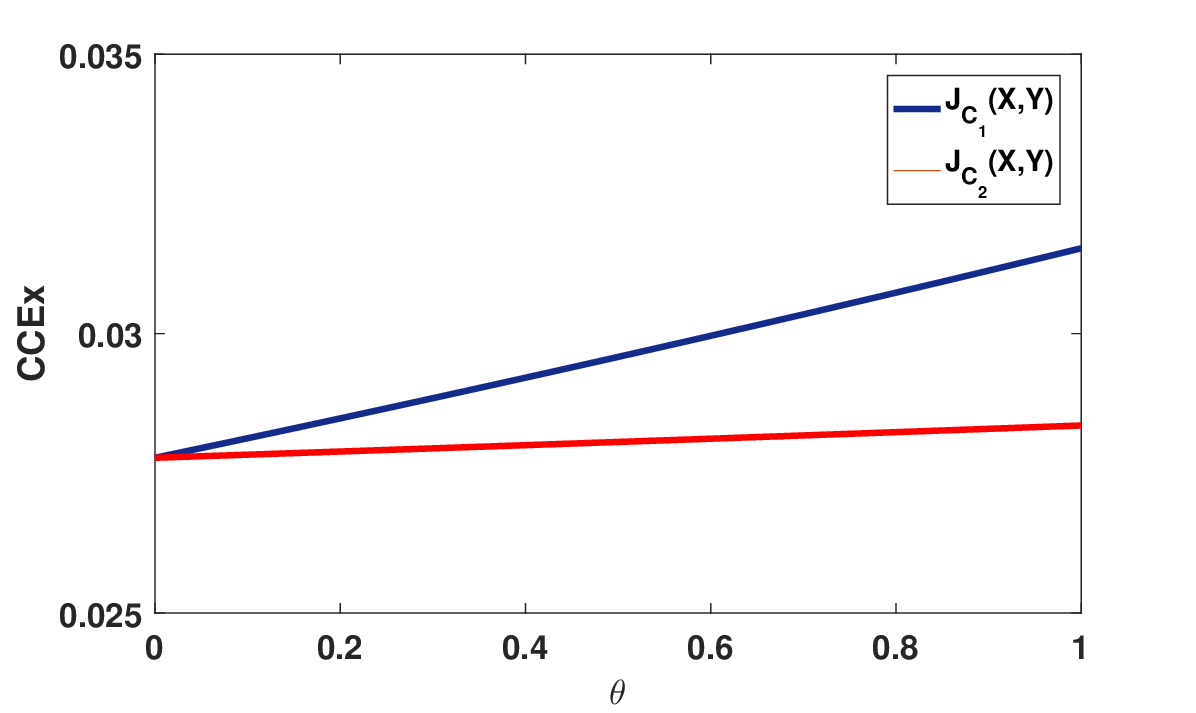}}
	\subfigure[]{\label{c1}\includegraphics[height=1.5in]{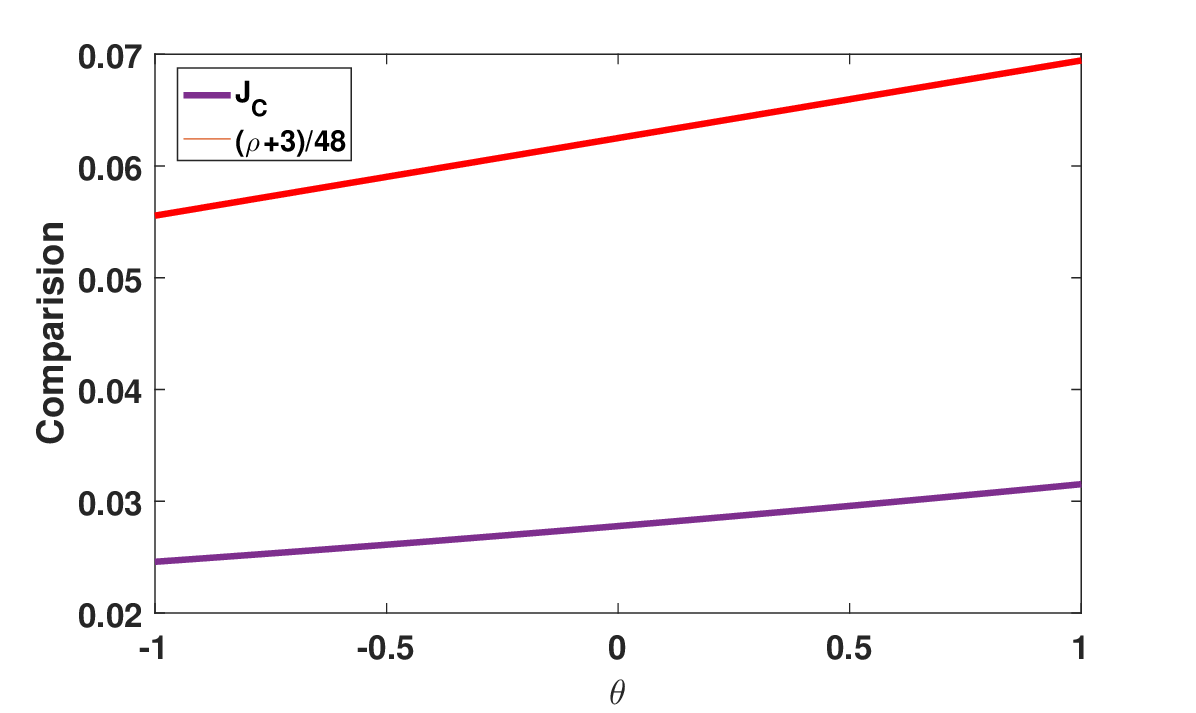}}
	\caption{$(a)$ Graphs for CCEx for extended FGM family with $p=1$ and  $p=2$ with respect to $\theta\in[0,1]$ in Example \ref{ex5.1}. $(b)$ Graphs for CCEx and $\frac{1}{48}(\rho+3)$ for FGM copula  with respect to $|\theta|\le1$ as in Example \ref{ex5.4}.}
\end{figure}

Further, we consider an example for the purpose of illustration of Proposition \ref{prop5.4}.

{\begin{example}\label{ex5.2}
		Consider Cuadras-Auge family of survival copulas 
		$$\bar C^\alpha(u,v)=
		\begin{cases}
		u+v-1+(1-u)(1-v)^{1-\alpha}, &~u\ge v;
		\\
		u+v-1+(1-u)^{1-\alpha}(1-v), &~ v\ge u,\\
		\end{cases}$$
		where $\alpha\in[0,1]$. For  $\alpha,\beta\in[0,1]$ and $\alpha\ge \beta$, it can be established that $\bar C^\alpha(u,v)\ge \bar C^{\beta}(u,v)$. The SCEx of the Cuadras-Auge family is obtained as
		\begin{align}\label{eq5.2}
		J_{\bar C^\alpha}=	\frac{1}{12}\bigg[\frac{1}{2}+\frac{2\alpha-3}{(2-\alpha)(3-\alpha)}+\frac{1}{(3-2\alpha)}\bigg].
		\end{align}
		The numerical values of $J_{\bar C^\alpha}$ given by (\ref{eq5.2}) are provided in Table $9$ for different values of $\alpha$, which clearly suggests that $J_{\bar C^\beta}\le J_{\bar C^\alpha}.$
		
				\begin{table}[h!]
			\caption { The numerical values of SCEx of Cuadras-Auge family of survival copulas for different values of $\alpha$ in $[0,1]$ in Example \ref{ex5.2}.}
			\centering 
			\scalebox{0.94}{\begin{tabular}{c c c c c c c c c c } 
					\hline\hline\vspace{.1cm} 
					$\alpha$ & 0.1& 0.2 & 0.3& 0.4& 0.5& 0.6& 0.7& 0.8& 0.9 \\
					\hline
					$J_{\bar C^\alpha}$ & 0.02908 & 0.03073 & 0.03281& 0.03547& 0.03889& 0.04332& 0.04916& 0.05699& 0.06782	\\	[1EX]
					\hline
			\end{tabular}} 
			\label{tb1} 
		\end{table}
\end{example}}

Suppose $\bar C_\theta(u,v)$ is a specified family of copulas indexed by a parameter $\theta\in \Theta$ and for all $u,v\in[0,1]$. The totally ordered  family $\{\bar C_\theta(u,v)\}, \theta\in\Theta$ is positively (negatively) ordered whenever $\bar C_{\theta_1}\prec(\succ)\bar C_{\theta_2}$ for $\theta_1\le \theta_2$, $\theta_1,\theta_2\in\Theta$. It is also true that $\bar C_{\theta_1}\prec(\succ)\bar C_{\theta_2} \Leftrightarrow  C_{\theta_1}\prec(\succ) C_{\theta_2}$. For details, see \cite{nelsen2006introduction}, p. $39$.

\begin{proposition}\label{pro5.4}
	We have $J_{C_{\theta_1}}\ge(\le)J_{ C_{\theta_2}}$, for a positively (negatively) ordered family of copulas $\{ C_{\theta}\}.$
\end{proposition}
\begin{proof}
	The proof is straightforward, and hence it is omitted. 
\end{proof}

The following example justifies Proposition \ref{pro5.4}.

\begin{example}
		For Nelsen's copula given in Table $1$, it is clear that $\theta_1\ge \theta_2$ implies $C_{\theta_1}\ge C_{\theta_2}$.
		We obtain
		\begin{align}\label{eq5.3*}
		J_{C_\theta}=\frac{1}{4}\left[\frac{1}{9}+0.1389\theta+2.58\theta^2\right].
		\end{align}
		From (\ref{eq5.3*}), we can easily observe that $J_{C_{\theta_1}}\ge J_{C_{\theta_2}},$ for $\theta_1\ge\theta_2$.
\end{example}

\begin{remark}
	A result similar to Proposition \ref{pro5.4} can be obtained for the case of SCEx.
\end{remark}

Now, we recall some dependence measures based on copula, and then establish relations with CCEx. Suppose $(X,Y)$ is a bivariate random vector with the corresponding copula $C(u,v)$ for all $u,v\in[0,1]$. Then, 
\begin{itemize}
	\item Spearman's rho: $\rho=12\int_{0}^{1}\int_{0}^{1}C(u,v)dudv-3$;
	\item Kendall's tau: $\tau=4\int_{0}^{1}\int_{0}^{1}c(u,v)C(u,v)dudv-1$;
	\item Blest's measure of rank correlation: $\eta=24\int_{0}^{1}\int_{0}^{1}(1-u)C(u,v)dudv-2$.
\end{itemize}
\cite{daniels1950rank} established a relation between $\rho$ and $\tau$ as follows
\begin{align}\label{eq5.3}
-1\le3\tau-2\rho\le1.
\end{align}
\cite{durbin1951inversions} demonstrated another second set of universal inequalities relating $\rho$ and $\tau$, given by 
\begin{equation} \label{eq5.4}
\left.
\begin{array}{ll}
& 1+\rho\ge\frac{(1+\tau)^2}{2}
\\
\
& 1-\rho\ge\frac{(1-\tau)^2}{2}
\end{array}
\Bigg\}\right.
\end{equation}

In the following, we obtain relations between CCEx and Spearman's rho, Kendall's tau and Blest's measure of rank correlation.			
\begin{proposition}\label{pro5.7}
	Suppose $C(\cdot,\cdot)$ is  a copula corresponding to $(X,Y)$. Then, 
	\begin{align}\label{eq5.5}
	J_C\le \frac{1}{48}(\rho+3).
	\end{align}
\end{proposition}			

\begin{proof}
	We have
	\begin{align}\label{eq5.6}
	C(u,v)\ge C^2(u,v)
	\Rightarrow \frac{1}{4}\int_{0}^{1}\int_{0}^{1}C^2(u,v)dudv\le\frac{1}{4}\int_{0}^{1}\int_{0}^{1}C(u,v)dudv.
	\end{align}
	Further, 
	\begin{align}\label{eq5.7}
	\int_{0}^{1}\int_{0}^{1}C(u,v)dudv=\frac{1}{12}(\rho+3).
	\end{align}
	Using (\ref{eq5.6}) and (\ref{eq5.7}), we get (\ref{eq5.5}), as desired.
\end{proof}

The following example validates Proposition \ref{pro5.7}.
\begin{example}\label{ex5.4}
	For FGM copula, we obtain  $$J_C=\frac{1}{4}\left(\frac{1}{9}+\frac{\theta}{72}+\frac{\theta^2}{900}\right)~\mbox{and}~\rho=\frac{\theta}{3}.$$
	Now, from Figure $4(b)$, the validation of  Proposition \ref{pro5.7} is clear.
\end{example}

\begin{proposition}\label{prop5.8}
	Let $(X,Y)$ be a random vector with copula $C(u,v)$. Then,
$$J_C\le \min\left\{\frac{1}{96}(3\tau+7), \frac{1}{96}[8-(1-\tau)^2]\right\}.$$
\end{proposition}
\begin{proof}
 From (\ref{eq5.3}), we have
	\begin{align}\label{eq5.8}
&\frac{1}{96}(3\tau+7)\ge \frac{1}{48}(\rho+3)\ge \frac{1}{96}(3\tau+5).
	\end{align}
	Now, from (\ref{eq5.5}) and (\ref{eq5.8}), we get 
	\begin{eqnarray}\label{eq5.9*}
	J_{C}\le \frac{1}{96}(3\tau+7).
	\end{eqnarray}
Further,  from (\ref{eq5.4}) we obtain 
\begin{align}\label{eq5.9}
 \frac{1}{48}(\rho+3)\le \frac{1}{96}[8-(1-\tau)^2].
\end{align}
Moreover, using (\ref{eq5.9}) in (\ref{eq5.5}), we get
\begin{align}\label{eq5.11*}
J_C\le\frac{1}{96}[8-(1-\tau)^2].
\end{align}
Now, the desired result follows from (\ref{eq5.9*}) and (\ref{eq5.11*}). This completes the proof.
\end{proof}

\begin{proposition}\label{prop5.7}
	Suppose $(X,Y)$ is a random vector with copula $C(\cdot,\cdot)$. Then, 
	\begin{align}\label{eq5.10}
	J_C-J^u_C\le\frac{1}{96}(\eta+2),
	\end{align}
	where $J^u_C(X,Y)=\frac{1}{4}\int_{0}^{1}\int_{0}^{1}uC(u,v)dudv$  is called weighted CCEx with weight function $u$.
\end{proposition}
\begin{proof}
	From the expression of Blest's measure of rank correlation, we have
	\begin{align}\label{eq5.11}
	 \int_{0}^{1}\int_{0}^{1}(1-u)C(u,v)dudv=\frac{1}{24}(\eta+2).
	\end{align}
	Again, 
	\begin{align*}
	& C(u,v)\ge C^2(u,v)\\
	\Rightarrow& \frac{1}{4}\int_{0}^{1}\int_{0}^{1}(1-u)C^2(u,v)dudv\le \frac{1}{4}\int_{0}^{1}\int_{0}^{1}(1-u)C(u,v)dudv\\
	\Rightarrow& \frac{1}{4}\int_{0}^{1}\int_{0}^{1}C^2(u,v)dudv-\frac{1}{4}\int_{0}^{1}\int_{0}^{1}uC^2(u,v)dudv\le\frac{1}{96}(\eta+2),~~~\text{[Using (\ref{eq5.11})]},
	\end{align*}
	proving the required result.
\end{proof}

The next example validates Proposition \ref{prop5.7}.

\begin{example}\label{ex5.5}
	Consider Cauadras-Auge family of copulas, where we take $\text{min}(u,v)=u$ for illustration purpose. Thus, the copula reduces to $C(u,v)=uv^{1-\alpha},~\alpha\in[0,1]$. Further, we obtain 
	\begin{eqnarray}
	 J_C&=&\frac{1}{12(2-\alpha)};\\
	J^u_C&=&\frac{1}{12(3-\alpha)};\\
	\eta&=&\frac{4}{2-\alpha}-2.
	\end{eqnarray}
	The numerical values of $J_C-J^u_C$ and $\frac{1}{96}(\eta+2)$ for different values of $\alpha$ are provided in Table $10$, validating (\ref{eq5.10}).
	
	\begin{table}[h!]
	\caption {Values of  $J_C-J^u_C$ and $\frac{1}{96}(\eta+2)$ for the copula in Example \ref{ex5.5}.}
	\centering 
	\scalebox{.9}{\begin{tabular}{c c c c c c c c c c} 
			\hline\hline\vspace{.1cm} 
			$\alpha$ & $0.1$ & $0.2$ & $0.3$& $0.4$& $0.5$ & $0.6$ &$0.7$ & $0.8$ & 0.9\\
			\hline
			$J_{ C}-J^u_C$ &$0.01512$ &$0.01653$ &$0.01815$& $0.02003$&$0.02222$&$0.02480$&$0.02787$& $0.03157$ & $0.03608$\\[1EX]
			\hline
			$\frac{1}{96}(\eta+2)$& $ 0.02193$ & $0.02315$& $0.02451$& $0.02604$& $0.02778$& $0.02976$ & $0.03205$& $0.03472$ & $0.03788$\\[1EX]
			\hline
	\end{tabular}} 
\end{table}
\end{example}

\begin{remark}
	Based on the survival copula function, the Spearman's rho, Kendall's tau, and Blest's measure of rank correlation can be defined respectively as follows:
	\begin{eqnarray}
		\rho&=&12\int_{0}^{1}\int_{0}^{1}\bar C(u,v)dudv-3;\label{eq5.17}\\
		\tau&=&4\int_{0}^{1}\int_{0}^{1}\bar c(u,v)[\bar C(u,v)-u-v+1]dudv-1;\label{eq5.18}\\
		\eta&=&24\int_{0}^{1}\int_{0}^{1}u\bar C(u,v)dudv-4.
	\end{eqnarray}
	Using (\ref{eq5.17}) and (\ref{eq5.18}), similar inequalities in Propositions  \ref{pro5.7} and \ref{prop5.8} can be obtained. 
\end{remark}

\section{An application}
In the preceding sections, we have discussed various properties of CCEx and SCEx. In the present section, we provide an application of CCEx and SCEx by proposing estimator using  empirical copula and empirical survival copula. We call the proposed estimators as the resubstitution estimators.
Suppose $\{(x_r,y_r)\}^n_{r=1}$ is a random sample of size $n$ from a continuous bivariate distribution. Then, the empirical copula and empirical survival copula are defined as 
\begin{align}\label{eq6.1}
C\left(\frac{i}{n},\frac{j}{n}\right)=\frac{\text{(number of pairs in the sample with $x\le x_{(i)}, y\le y_{(j)}$)}}{n}
\end{align}
and 
\begin{align}\label{eq6.2}
\bar C\left(\frac{i}{n},\frac{j}{n}\right)=\frac{\text{(number of pairs in the sample with $x> x_{(i)}, y> y_{(j)}$)}}{n},
\end{align}
respectively, where $x_{(i)}$ and $y_{(j)},$ for $1\le i,j\le n$ denote the ordered statistics of the sample. Please refer to \cite{nelsen2006introduction} for more discussion about this topic.
\begin{definition}\label{def6.1}
Suppose $(X,Y)$ is a random vector with corresponding empirical copula and survival copula $C\left(\frac{i}{n},\frac{j}{n}\right)$ and $\bar C\left(\frac{i}{n},\frac{j}{n}\right)$, respectively. Then, the resubstitution or plug-in estimator of
\begin{itemize}
\item CCEx is  defined as 
\begin{align}
\hat{J}_C(X,Y)=\frac{1}{670n^2}\sum_{i=1}^{n}\sum_{j=1}^{n}C^2\left(\frac{i}{n},\frac{j}{n}\right);
\end{align}
\item  SCEx is defined as 
\begin{align}
\hat{J}_{\bar C}(X,Y)=\frac{1}{670n^2}\sum_{i=1}^{n}\sum_{j=1}^{n}\bar C^2\left(\frac{i}{n},\frac{j}{n}\right).
\end{align}
\end{itemize}
\end{definition}

Now, we consider a real life data set for computing the values of the estimators proposed in Definition \ref{def6.1}. The datasets are available in \cite{kumar2007copula}. The data set represents  to an investigation on 20 individuals for isolated aortic regurgitations
before and after surgery and 20 persons for isolated mitral regurgitation. Namely data $(X)$ and  $(Y)$ denotes as pre-operative ejection fraction and post-operative ejection fraction respectively. The data set is given below:

\begin{table}[h!]
			\centering 
			\scalebox{0.90}{\begin{tabular}{|c |c|  } 
					\hline 
					$X$& 0.29,~ 0.36,~ 0.39,~ 0.41,~ 0.50,~ 0.53,~ 0.54, 0.55,~ 0.56,~ 0.56,~ 0.56,~ 0.58,~ 0.60,~ 0.60,~ 0.62,\\
					~&~ 0.64, 0.64,~ 0.67,~ 0.80,~ 0.87\\
					\hline
					$Y$& 0.17,~ 0.24,~ 0.26,~ 0.26,~ 0.27,~ 0.29,~ 0.30,~ 0.32, ~0.33, ~0.33,~ 0.34,~ 0.38,~ 0.47, ~0.47,~ 0.50,\\
					~&~0.56, ~0.58, ~0.59,~ 0.62, ~0.63	\\	[1EX]
					\hline
			\end{tabular}} 
			\label{tb1} 
		\end{table}

The estimated values are $\hat{J}_{\bar C}(X,Y)=0.05691$ and $\hat{J}_{ C}(X,Y)=0.07683$. The data set gives satisfactory evidence for the PQD property of the underlying copula since clearly
$$C\left(\frac{i}{n},\frac{j}{n}\right)\ge \frac{i}{n}\times \frac{j}{n}.$$
Based on the data set,  the Pearson correlation
coefficient $r = 0.6870$, Kendall’s rank correlation $\tau = 0.5050$ and Spearman’s
rank correlation $\rho = 0.6970$ are obtained by \cite{kumar2007copula} . Thus all the measures indicate positive dependence
between the two sets of observations.
We observe that the estimated values satisfy the results in Proposition \ref{prop5.1}, Proposition \ref{pro5.7}, Proposition \ref{prop5.8} and  Remark \ref{rem5.1}.

\section{Conclusion}
In this work, the information measures: copula extropy, CCEx and SCEx have been proposed based on density copula, copula and survival copula, respectively. We have studied various properties of the proposed information measures. We have discussed the effectiveness for applying monotone transformation.  Other information measures: horizontal, vertical and diagonal copula extropy based on sections of copula function have been provided and studied. The dual copula extropy and co-copula extropy have been introduced. In addition,   we have established  relations between cumulative copula extropy and survival copula extropy respectively. Further, some relations have been established between the proposed measures  and some dependence measures: Spearman's rho, Kendall's tau, and Blest's measure of rank correlation. Finally, for solving practical problems, we have introduced resubstitution estimators of CCEx and SCEx.

\section*{Acknowledgements}  The author Shital Saha thanks the UGC (Award No. $191620139416$), India, for the financial assistantship.

\section*{Conflicts of interest} Both authors declare no conflict of interest. 

\bibliography{refference}
\end{document}